\documentclass[12pt]{amsart}

\usepackage[notref,notcite]{showkeys}

\usepackage{amsmath,amstext,amssymb,amsopn,amsthm}
\usepackage{url,verbatim}
\usepackage{mathtools}
\usepackage{enumerate}

\usepackage{color,graphicx}

\usepackage[margin=30mm]{geometry}
\usepackage{eucal,mathrsfs,dsfont}

\RequirePackage[colorlinks,citecolor=blue,urlcolor=blue]{hyperref}

\allowdisplaybreaks

\newtheorem{theorem}{Theorem}[section]
\newtheorem{corollary}[theorem]{Corollary}
\newtheorem{lemma}[theorem]{Lemma}
\newtheorem{proposition}[theorem]{Proposition}

\theoremstyle{definition}

\newtheorem{definition}[theorem]{Definition}

\newtheorem{remark}[theorem]{Remark}
\newtheorem{example}[theorem]{Example}

\numberwithin{equation}{section}

\newcommand{\eps}{\varepsilon}
\newcommand{\kf}{{{\mathbf k}(f)}}

\newcommand{\calL}{\mathcal{L}}
\newcommand{\calF}{\mathcal{F}}

\newcommand{\calG}{\mathcal{G}}
\newcommand{\calA}{\mathcal{A}}

\newcommand{\calC}{\mathcal{C}}
\newcommand{\calE}{\mathcal{E}}

\newcommand{\calV}{\mathcal{V}}
\newcommand{\calH}{\mathcal{H}}
\newcommand{\calN}{\mathcal{N}}

\newcommand{\bM}{\mathbf{M}}

\renewcommand{\P}{\operatorname{\mathds{P}}} 

\newcommand{\prt}{\partial}
\newcommand{\lra}{\leftrightarrow}

\newcommand{\bL}{\mathbf{L}}

\newcommand{\ol}{\overline}
\newcommand{\ul}{\underline}
\newcommand{\wh}{\widehat}
\newcommand{\wt}{\widetilde}

\DeclareMathOperator{\dist}{dist}

\DeclareMathOperator{\sign}{sign}

\def\bone{{\bf 1}}
\def\n{{\bf n}}

\def\bv{{\bf v}}
\def\bw{{\bf w}}

\renewcommand{\comment}[1]{}
\newcommand{\eq}{\begin{equation}}
\newcommand{\en}{\end{equation}}

\newcommand{\abs}[1]{\left\lvert #1 \right\rvert}

\title{Floodings of metric graphs}
\author{Krzysztof Burdzy and Soumik Pal}

\address{ Department of Mathematics, Box 354350, University of Washington, Seattle, WA 98195}
\email{burdzy@uw.edu}

\email{soumik@uw.edu}

\thanks{Research supported in part by NSF Grant DMS-1206276 and DMS-1612483. }

\begin{document}

\begin{abstract}
We consider random labelings of finite graphs conditioned on a small fixed number of peaks. We introduce a continuum framework where a combinatorial graph is associated with a metric graph and edges are identified with intervals.  
Next we consider a sequence of partitions of the edges of the metric graph with the partition size going to zero. As the mesh of the subdivision goes to zero, the conditioned random labelings converge, in a suitable sense, to a deterministic function which evolves as an increasing process of subsets of the metric graph that grows at rate one while maximizing an appropriate notion of entropy. We call such functions floodings. We present a number of qualitative and quantitative properties of floodings and some explicit examples. 
\end{abstract}

\maketitle

\section{Introduction}\label{intro}

Consider a finite graph with $N$ vertices. Randomly label the vertices by numbers $1,2,\ldots,N$, conditioned to have a fixed number $M$ of peaks (local maxima). The problem, a continuation of the project started in \cite{twinp}, is to understand the most likely locations of the peaks as a function of the underlying graph structure. Even for the case of $M=1$ or $M=2$, simple questions seem to be very hard to answer for random labelings on many natural graphs. The following is an example that was left open in \cite{twinp}. If a random labeling of a large torus is conditioned on having exactly two peaks, is the typical distance between the peaks comparable to the diameter of the torus? 

In this article we introduce a limiting procedure where asymptotically the labelings are amenable to analytical methods. Let us describe the procedure very roughly for now. All details can be found in Section \ref{subdiv}. Imagine that our graph has edge lengths and can be viewed as a metric space. Replace every edge by an interval and consider a sequence of finer and finer subdivisions of each interval. Thus for every metric graph we have a sequence of subdivision metric graphs with the number of vertices going to infinity that can be embedded in the original graph. Now, we fix $M$ and consider the sequence of random labelings on each subdivision graph, conditioned to have $M$ peaks, and take a limit as the number of vertices goes to infinity.

In an appropriate sense, the labelings converge to a continuum limit that we call ``flooding''.  Again, let us give a rough intuition regarding this limiting object. Consider again the random labeling of the finite graph with $N$ vertices. This static random object can be turned dynamic by considering a time parameter $t\in [0,1]$ and progressively exploring the subgraphs labeled by vertices $\{ \lfloor Ns\rfloor,\; (1-t)\le s\le 1  \}$. A moments reflection will convince the reader that, due to our constraint on having $M$ peaks, the stochastic process of subgraphs is increasing and yet always remains a union of at most $M$ connected components.  
This property persists in the subdivision limit where we obtain a stochastic process of closed subsets of the metric graph that is a union of at most $M$ connected components and ``absorbs'' boundary points at rate one until it covers the entire graph. An evocative visual would be to imagine the edges as pipes and water flowing through $M$ sources until it floods the entire graph. This inspires the name ``flooding''. 

Interestingly, the floodings on metric graphs are surprisingly rigid objects. The reason, in short, is large deviation theory. Each flooding can be studied by an entropy-like quantity (see \eqref{f5.1}) and the asymptotically most likely flooding is the one that maximizes this entropy. This is essentially the content of Theorem \ref{j22.1}. The rest of the paper analyzes the maximal-entropy floodings on trees and non-tree graphs and shows some remarkable properties (see Theorem \ref{f22.1}). For example, we show that, on arbitrary graphs, the optimal flooding, until the very end, has exactly $M$ connected components. That is to say, the floodings initiated from separate sources will never meet until the end. The optimal sources are given by centroids of metric trees.

The paper is organized as follows. It starts with Section \ref{prel} containing some preliminaries.   Section \ref{subdiv} will 
introduce subdivision graphs and floodings, and present an asymptotic result on random labelings. Properties of floodings will be studied in Section \ref{flood}. Finally, examples of floodings will be presented and examined in 
Section \ref{sec:example}.

\section{Preliminaries}\label{prel}

For an integer $N>0$, let $[N] =\{1,2,\dots, N\}$.
By abuse of notation, we will use $|\,\cdot\,|$  to denote the absolute value of a real number, cardinality of a finite set and length of a line segment.
For a real number $a$, we will write $ \lfloor a\rfloor $ to denote  the largest integer less than or equal to $a$ and $  \lceil a \rceil$ to denote the smallest integer greater than or equal to a.

We will usually denote graphs $\calG$, their vertex sets $\calV$ and edge sets $\calE$.
We will indicate adjacency of vertices $v$ and $y$ by $v\lra y$.
We will consider only finite simple graphs, i.e., graphs with no loops and multiple edges.

\subsection{Labelings}

 We will call a function $L: \calV \to [N]$ a labeling of $\calG$ or labeling of $\calV$ if it is a bijection. 
The inverse function will be denoted $R$.

We will say that a vertex $v\in \calV$ is a peak (of $L$) if and only if it is a local maximum of $L$, i.e., $L(v) > L(y)$ for all 
$y\lra v$.

We will use $\P$ to denote the distribution of a random (uniform, unconditioned) labeling. The symbol $\P_M$  will stand for the distribution $\P$ conditioned on existence of exactly $M$ peaks. 

Consider integers $M,N>0$ and a graph $\calG$ with $N$ vertices. We will describe the general structure of a (deterministic) labeling $L$ with exactly $M$ peaks. 
Let $v_1, v_2, \dots , v_M \in \calV$ be the peaks of $L$ and let $a_k = L(v_k)$ for $1\leq k \leq M$. Note that $v_m \not \lra v_\ell$ for $m\ne \ell$.
Assume without loss of generality that $a_1 < a_2 < \dots < a_M$.
The vertex $R(N)$ is a peak so $a_M=N$. If $N-1 \ne a_{M-1}$, i.e., $R(N-1)$ is not a peak, then we must have $R(N-1) \lra R(N)$. 
We will generalize this remark.
Let $C_k = \{R(N), R(N-1), \dots, R(N-k)\}$ for $0\leq k \leq N-1$.
If $1 \leq N-m < N$ and $N-m \ne a_\ell$ for all $1\leq \ell \leq M$, i.e., $R(N-m)$ is not a peak, then $R(N-m) \lra v$ for some $v\in C_{m-1}$.
If $1 \leq N-m < N$ and  $R(N-m)$ is a peak then  $R(N-m) \not \lra v$ for all $v\in C_{m-1}$. 

The above analysis suggests the following dynamic picture of a random labeling with one peak. Conditional on the position of the peak, all other values of the labeling should be chosen (randomly) by attaching them, from the highest to the lowest value, one by one, to the clusters $C_k$ of already occupied vertices. Intuition might suggest that $R(N-m)$ should be chosen uniformly from all  vertices adjacent to $C_{m-1}$. That this is not the case
can be demonstrated using very simple examples, see \cite[Example 2.1]{twinp}. 
A much more interesting example is provided by the comparison of computer
simulations of the Eden model (where new labels are attached uniformly on the boundary) and a random labeling conditioned on having a single peak, on a large discrete torus. The simulations presented in \cite[Sec. 2]{twinp} and a rigorous result, Theorem 4.2 in \cite{twinp}, suggest that the two mechanisms generate significantly different labelings (belonging to different ``universality'' classes).

\subsection{Large deviations for multinomial distribution}
As mentioned before, our results on floodings are closely related to the large deviations theory for multinomial distribution; see, e.g., \cite{DZ}. Although this type of results are classical, they provided little help with the current project. It was easier to repeat some of the classical computations rather than to reduce our claims to those that can be found in  the existing literature. This section contains a review of some elementary properties of the multinomial distribution, presented in a form that is convenient for the applications in this paper.

Consider a  graph $\calH$ consisting of $K$ disjoint subgraphs $\calH_j$, $1\leq j \leq K$. Suppose that each  of the subgraphs $\calH_j$ is isomorphic to the set of natural numbers (with edges  connecting neighbors). Let $\calV$ ($\calV_j$) denote the vertex set of $\calH$ ($\calH_j$, resp.) and let $v_1, \dots, v_K$ be the endpoints of the subgraphs $\calH_j$. 
Suppose that $N>0$  and  $N= N_1 + N_2 + \dots + N_K$ where each $N_j$ is a non-negative integer.

\begin{lemma}\label{lem:multinomialcount}
(i) Let $\calN$ be the number 
of (deterministic) functions $R: [N] \to \calV$ satisfying the following conditions: (a) $R$ is injective, (b) $|R([N])\cap \calV_j| = N_j$ for $j\in [K]$, (c) $R([N])\cap \calV_j$ is connected for each $j$, and (d) $v_1, \dots, v_K$ are the only local maxima of $R^{-1}$. 

Suppose that for some non-negative real numbers $a_1, a_2,\dots, a_K, a$ and integers $n$ and $K_1$ we have
$|N_j - n a_j |\leq K_1$ for $j\in[K]$, and
$a_1+ a_2+\dots+ a_K= a$. 
Then for any fixed $K$, $K_1$ and a real number $\zeta>0$, uniformly in $a\in[0,\zeta]$ and $a_j$'s, when $n\to \infty$,
\begin{align}\label{m24.1}
\log \calN 
&= N \log N 
-\sum_{j=1}^K N_j \log N_j + O(\log N)\\
&= na \log (na) 
- \sum_{j=1}^K n a_j \log (n a_j) + O(\log n) \nonumber \\
&= n \left(a  \log a
- \sum_{j=1}^K  a_j  \log a_j\right) + O(\log n). \nonumber 
\end{align}


(ii) Now suppose that $R$ is a uniformly chosen random element  from the set of functions described in (i) above. 
Then
for every $\eps >0$ there is $N_*$ such that for $N\geq N_*$ and all non-negative integers $N_1, N_2, \dots , N_K$ such that $N= N_1 + N_2 + \dots + N_K$, we have
\begin{align}\label{m30.2}
\P\left( \sup_{0< t \leq 1}\big||R([ \lceil N t\rceil ])\cap \calV_j| - N_j t\big| \leq \eps N \right)
\geq 1-\eps.
\end{align}

\end{lemma}

\begin{proof}
(i)
The sets $R([N])\cap \calV_j$, $j\in [K]$, form a  partition of $[N]$ into $K$ ordered subsets of cardinalities $N_1, N_2, \dots , N_K$. Hence, the number $\calN$
of functions $R$ satisfying the  conditions in part (i) of the lemma is 
\begin{align}\label{f5.4}
\calN
= \calN( N; N_1, N_2, \dots , N_K)
=\binom{N}{N_1, N_2, \dots , N_K}
= \frac {N!}{N_1! N_2! \dots N_K!}.
\end{align}
Recall the Stirling formula, 
$\log n! = n \log n - n + O(\log n)$.
It follows from this formula and \eqref{f5.4} that, for a fixed $K$, 
\begin{align*}
\log \calN &= N \log N - N + O(\log N)
- \sum_{j=1}^K ( N_j \log N_j - N_j + O(\log N_j))\\
&= N \log N 
- \sum_{j=1}^K N_j \log N_j + O(\log N) .
\end{align*}
Our assumptions imply that $|N - n a |\leq KK_1$.
Hence for any fixed $K$, $K_1$ and a real number $\zeta>0$, uniformly in $a\in[0,\zeta]$ and $a_j$'s, when $n\to \infty$,
\begin{align*}
\log \calN 
&= N \log N 
-\sum_{j=1}^K N_j \log N_j + O(\log N)\\
&= na \log (na) 
- \sum_{j=1}^K n a_j \log (n a_j) + O(\log n) \nonumber \\
&= n \left(a  \log a
- \sum_{j=1}^K  a_j  \log a_j\right) + O(\log n). \nonumber 
\end{align*}


(ii) This part follows easily from the law of large numbers. More precisely, consider a uniformly random arrangement of $N$ labeled balls in $K$ labeled boxes such that box $j$ has $N_j$ balls. Sequentially remove balls labeled $N, N-1, \ldots$ at unit intervals of time. If we now reverse time, the law of the process $\left(|R([ \lceil N t\rceil ])\cap \calV_j|, t\ge 0,\; j=1,2,\ldots, K\right)$ is identical to that of the process of the number of balls in each box at time $t$. 

\end{proof}

\section{Metric graphs and floodings}\label{subdiv}

\subsection{Metric graphs}

We start this section with a definition of a metric graph. Let $\calG=(\calV,\calE)$ be a finite connected simple graph. A metric graph corresponding to $\calG$ is a pair $(\calG,\ell)$ where $\ell: \calE \rightarrow (0, \infty)$ assigns a positive length to each edge of $\calG$. It is intuitive to visualize the metric graph by replacing each edge $e$ by an interval $[0, \ell(e)]$.  More precisely, consider an arbitrary direction for each edge. That is, in an arbitrary manner replace every edge $e=\{u,v\} \in \calE$ by a vector $e=(u,v)$ where, we say, $u$ is the initial vertex of the edge and $v$ is the terminal vertex. This changes the meaning of $\calE$, but by an abuse of notation we continue to use $\calE$ for this set of directed edges. Note that, if $(u,v)\in \calE$ then $(v,u) \notin \calE$. We say, vertices $u$ and $v$ share an edge if either $(u,v)$ or $(v,u)$ is in $\calE$. 

Now, consider the following set 
\[
S^*= \bigcup_{e \in \calE} \left\{ \left\{ e \right\} \times [0, \ell(e)] \right\}.
\] 
The vertex set $\calV$ can be recovered from $S^*$ via the following equivalence relation. For $e,g \in \calE$ and $x\in [0, \ell(e)]$ and $y \in [0, \ell(g)]$, we define $(e,x) \sim (g,y)$ if, either $e=g$ and $x=y$, or $e\neq g$ but share a vertex, and $x$ and $y$ represent the same vertex. This can happen in the following four ways. 
\begin{enumerate}[(i)] 
\item $e=(u,v), g=(u,w)$ and $x=0=y$.
\item $e=(u,v), g=(w,u)$ and $x=0, y=\ell(g)$.
\item $e=(u,v), g=(w,v)$ and $x=\ell(e), y=\ell(g)$.
\item $e=(u,v), g=(v,w)$ and $x=\ell(e), y=0$.
\end{enumerate}
It is not hard to see that $\sim$ is an equivalence relation on $S^*$. Consider the quotient set $S=S^*/\sim$. We will continue to represent points in $S$ by any representative element $(e,x)$ in $S^*$ from its equivalent class. Then, the set $\calV$ of vertices are precisely those points $(e,x)\in S$ such that $x\in \{0, \ell(e)\}$. Let $\bv(e,x)$ denote the equivalence class of $(e,x)\in S$.  We call $S$ a real metric graph.  

The notion of degree can now be extended to points in $S$. By definition, any point $(e,x)$ for $0 < x < \ell(e)$ has degree $2$. Any other point will correspond to a vertex $v \in \calV$. The degree of that point will be the degree of $v$ in the graph $\calG$. 

There is a natural graph metric on the set $\calV$ obtained from the metric graph $(\calG, \ell)$. To wit, consider two vertices $u,v$ in $\calV$. If $u=v$, the distance is zero. Otherwise, consider a path in $\calG$ from $u$ to $v$ in $k$ steps for some $k\ge 1$. This is a finite sequence of vertices $\{w_0, w_1, \ldots, w_k\} \subseteq \calV$ such that $w_0=u$, $w_k=v$ and each successive pair $w_i, w_{i+1}$ share an edge $e_i$, for $i=0,1,\ldots, k-1$. The length of this path is then defined to be $\sum_{i=0}^{k-1} \ell(e_i)$. The distance between $u$ and $v$, denoted by $\dist(u,v)$ is the infimum of the length of path taken over all paths of any number of steps connecting $u$ and $v$. We will now extend this metric to $S$. 

For any pair of points $(e,x)$ and $(g,y)$ the distance between them can be defined as follows. If $e=g$, then 
\[
\dist\left( (e,x), (g,y)\right)= \abs{x-y}. 
\]
Now, suppose $e\neq g$. Let $u_1=(e,0)$, $u_2=(e, \ell(e))$, $v_1=(g,0)$, $v_2=(g, \ell(g))$.
Then, $\dist\left( (e,x), (g,y)\right)$ is the minimum of the following four terms: (i) $x+\dist(u_1,v_1)+y$, (ii) $x + \dist(u_1, v_2)+  \ell(g)-y$, (iii) $\ell(e)-x + \dist(u_2, v_1) + y$, (iv) $\ell(e)-x + \dist(u_2, v_2) + \ell(g)-y$. 

We leave the verification of the metric property of $\dist(\,\cdot\,,\,\cdot\,)$ to the reader. The metric gives rise to the usual topology generated by open balls, and associated concepts, for example, that of the boundary of a set. The boundary of a set $A$ will be denoted in the usual way by $\prt A$.

\subsection{Subdivision graphs}\label{f24.1}

We will define a subdivision graph $\calG_n$ of a (metric) graph $\calG$, for $n\geq 1$.  
The graph $\calG_n$ is obtained by replacing every edge $e$ of $\calG$ with a path (linear) graph with $\lceil n \ell(e)\rceil$ edges. More precisely, if the endpoints of $e$ are $w_1$ and $w_2$ then we add vertices $v^e_1, v^e_2, \dots, v^e_{\lceil n \ell(e)\rceil-1}$ to the vertex set, we identify $v^e_0$ with $w_1$ and $v^e_{\lceil n \ell(e)\rceil}$ with $w_2$ and we add edges so that
$w_1= v^e_0 \lra v^e_1\lra \dots \lra v^e_{\lceil n \ell(e)\rceil} =w_2$.
We assume that the sets of extra vertices $\{ v^e_1, \dots, v^e_{\lceil n \ell(e)\rceil-1}\}$ are disjoint for distinct $e$. 
The vertex set of $\calG_n$ will be denoted $\calV_n$.
If an edge $g$ of $\calG_n$ has been obtained by subdividing an edge $e$ of $\calG$ then we give $g$ the length $\ell(g) = \ell(e) / \lceil n \ell(e) \rceil$.

It is easy to see that the metric graph $S$ associated with a graph $\calG$ is isometric (as a metric space) with the metric graph corresponding to any of the  subdivision graphs $\calG_n$ of $\calG$.

\subsection{Normal vectors}

For an edge $e $, we will define abstract ``unit vectors'' $ e^+$ and $e^-$.  
We will write $\sign\left(e^+\right) = 1$ and 
$\sign\left(e^-\right) = -1$.
We will use this concept only in the expressions of the form $(g,x) + t e^+$ or $(g,x) + t e^-$, where $(g,x) = (e,s)$ for some $0\leq s \leq \ell(e)$. Hence, all we need to do is to give a meaning to these expressions. If $\bw = e^+$ or $\bw = e^-$ then we declare that $(g,x) + t \bw = \left(e, s+\sign\left(\bw\right)t\right) $ for all $t$ such that $0 \leq s+\sign\left(\bw\right)t \leq \ell(e)$.

For any $(g,x)$,
we say that $e^+$ is a vector anchored at $(g,x)$ if $(g,x) = (e,s)$ for some $0\leq s < \ell(e)$.
We say that $e^-$ is  anchored at $(g,x)$ if $(g,x) = (e,s)$ for some $0< s \leq \ell(e)$. 

For any closed set $A$ with finitely many connected components and $(g,x)\in \prt A$, we say that $\bw$ is an outward normal vector to $A$ at $(g,x)$ if $\bw$ is a vector anchored at $(g,x)$
 and there exists $t_1>0$ such that $(g,x) + t \bw \notin A$ for  all $t\in(0, t_1)$.

\subsection{Floodings}

\begin{definition}\label{m5.1}
We  define \textit{flooding}  as a process $\left( f(t), \; t\ge 0   \right)$ of increasing closed subsets of $S$ such that $f(t)$ has a finite number of connected components for every $t\geq 0$. The process $f$ is defined as a piecewise linearly growing set in $S$. More precisely, for $f$ to be a flooding, there must exist a finite sequence of times $t_0=0< t_1 < t_2 <\dots < t_\kf < \infty$  such that the following conditions hold for $k=1,2,\ldots, \kf+1$. 
\begin{enumerate}[(i)]
\item  If  $f(t_{k-1})=S$ then $f(t)=S$ for all $t \ge t_{k-1}$ and $\kf = k-1$, i.e., $t_\kf = t_{k-1}$.

\item 
If  $f(t_{k-1})\neq S$ then $f(t_{k-1})$
is a finite union of closed connected sets. 
 Let $\left\{ \bw_j^k, j=1,2,\ldots, m_k \right\}$ denote the set of all outward normal vectors at all  points in $\prt f(t_{k-1})$
(so that $m_k$ denotes the total number of outward normal vectors). 
Every vector $\bw_j^k$ can be uniquely represented as $\bw_j^k = (e_j^k)^+$ or $\bw_j^k = (e_j^k)^-$ for some edge $e_j^k$.
Note that at least one of the representations of the point in $\prt f(t_{k-1})$ associated with $\bw_j^k$ must have the form $(e_j^k, r_j^k)$ for some $r_j^k \in[0, \ell(e_j^k)]$. We may have $(e_{j_1}^k, r_{j_1}^k) = (e_{j_2}^k, r_{j_2}^k)$ for some $j_1\ne j_2$.
\item For $t_{k-1}< t_\kf$, there exists a vector
$\left( z^k_{j}, \; 1\le j\le m_k   \right)$ in $(m_k-1)$-dimensional unit simplex. In other words, every $z^k_{j}$ is nonnegative and 
\begin{align}\label{f5.3}
\sum_{1\leq j \leq m_k} z^k_j = 1.
\end{align}

\item The time $t_k$ is the infimum of $t > t_{k-1}$ such that  for some  $j$ we have 
\begin{align*}
 r_j^k + \sign( \bw_j^k)z_j^k(t-t_{k-1}) \notin [0, \ell(e_j^k)],
\end{align*}
or, for some $j_1\ne j_2$ we have $e^k_{j_1} = e^k_{j_2}$ and 
\begin{align*}
r_{j_1}^k + \sign( \bw_{j_1}^k)z_{j_1}^k(t-t_{k-1})
=r_{j_2}^k + \sign( \bw_{j_2}^k)z_{j_2}^k(t-t_{k-1}).
\end{align*}

\item 
For $t\in[t_{k-1}, t_k]$, let $x_j^k(t) =  (e^k_j, r_j^k) +  (t-t_{k-1}) z^k_{j}  \bw _j^k$.
For $t\in[t_{k-1}, t_k]$, the flooding must satisfy
\eq\label{eq:linear_rate}
f(t) = f(t_{k-1}) \cup 
\left\{x_j^k(s):  1\le j \le m_k,\, 0\le s \le t-t_{k-1} \right\}.
\en

\end{enumerate}

\end{definition}

At a linear rate the set $f(\cdot)$ absorbs points on the boundary until it hits a vertex (which potentially alters the size of the set of outer normals) or an edge is absorbed completely by exhausting it from both sides (which reduces the boundary). Since $z^k_j$'s are always chosen from the simplex, the total rate of absorption is always one. Let $\zeta = \sum_{e\in \calE} \ell(e)$ be the total length of all edges of $S$. Hence, $t_\kf=\zeta<\infty$,  $f(t_\kf)=S$ and the process gets absorbed in this state for all subsequent times.  
In other words, condition \eqref{f5.3} agrees with the assumption that $\zeta = \inf\{t\geq 0: f(t) = S\}$.

Note that  the total number of all the outward normals vectors of all the boundary points of $f(t)$ is equal to $m_k$ for all $t_{k-1} \le t < t_k$. Informally, for $t\in(0, t_k - t_{k-1})$, we will refer to any set of the form  $ \left\{ (e^k_j, r_j^k) +  s z^k_{j}  \bw _j^k,\; 0\le s \le t \right\}$ as an arm of the flooding. 

We will write $z^k_j = \dot x^k_j(t)$ for $t\in(t_{k-1}, t_k)$. This is a shorthand for saying that $z^k_j$ is the derivative of $\dist(x^k_j(t_{k-1}), x^k_j(t))$ with respect to time $t$.


For any $M\geq 1$, we will say that $\left( f(t), \; t\ge 0    \right)$ is a flooding with $M$ \textit{sources} if $f(0)$ is a set of $M$ distinct points. The family of all floodings with $M$  sources will be denoted $\calF^M$. We will write $\calF^{\leq M}=\bigcup_{1\leq K \leq M} \calF^K$.

We will define a metric for the family of floodings of a given graph $\calG$.

We equip the family of non-empty compact subsets of $S$ with the usual Hausdorff distance, denoted $\dist_H$. It is easy to see that every flooding $f$ is Lipschitz with the Lipschitz constant 1, i.e., $\dist_H(f(s),f(t)) \leq t-s$ for all $0\leq s \leq t \leq \zeta$.
We define the distance between floodings $f$ and $\wt f$ by 
$\dist(f, \wt f) = \sup \{\dist_H(f(t), \wt f(t)): t\in[0,\zeta]\}$.

We will sketch a proof that $\calF^{\leq M}$ is compact for every $M$.
First, it is easy to see that the family of subsets of $S$ that have at most $M$ elements is compact in the topology induced by $\dist_H$ because $S$ is compact. Hence, for any sequence $f_j$ of floodings in $\calF^{\leq M}$, there is a subsequence $f_{j_i}$ such that $f_{j_i}(0)$ converges to a limit and the limit is a set with at most $M$ points.
The family of compact subsets of $S$ is compact in the metric $\dist_H$. Hence, for every rational number $t\in[0,\zeta]$, we can find a subsequence $f_{j_i}$ of $f_j$, such that $f_{j_i}(t)$ converges to a compact subset of $S$. By the diagonal method, we can find a single subsequence $f_{j_i}$ of $f_i$ such that $f_{j_i}(t)$ converges to a compact subset of $S$ for every rational $t\in[0,\zeta]$. By the Lipschitz property of floodings, the last claim holds for all real $t\in[0,\zeta]$. 

For any flooding $f$ we have $\kf\leq \n_\calG:= |\calV|+|\calE|$, i.e., $\kf$ is bounded by the number of vertices plus the number of edges of $\calG$. This is because at every time $t_k$, $1\leq k\leq \kf$, the flooding reaches at least one of the vertices for the first time or two arms of a flooding  meet in the interior of an edge. It is easy to see that each vertex and each edge can play this role at most once.

Let $t^i_0=0< t^i_1 < t^i_2 <\dots < t^i_{{\bf k}(f_{j_i})} =\zeta$ be defined relative to the flooding $f_{j_i}$. Passing to a subsequence, if necessary, we may assume that ${\bf k}(f_{j_i})$ are  equal to each other for all $i$, and a limit $\lim_{i\to \infty}t^i_k = t^\infty_k \in[0,\zeta] $ exists for every $k=1,\dots, {\bf k}(f_{j_i})$. It is not necessarily true that $t^\infty_j \ne t^\infty_k$ for $j\ne k$.

For every $\eps >0$ and $k>0$ there exists $i_1$ such that for all $i>i_1$ and $j$, $t^i_j \notin [t^\infty_{k-1} + \eps, t^\infty_k - \eps]$ (the claim holds trivially if the interval is empty). Hence, all floodings $f_{j_i}$, $i>i_1$, are linear on $[t^\infty_{k-1} + \eps, t^\infty_k - \eps]$ in the sense of the Definition \ref{m5.1} (v). This implies that the limiting function, say, $f_\infty$, is also linear on this interval. The claim can be extended to the interval $[t^\infty_{k-1} , t^\infty_k ]$ because $\eps>0$ can be taken arbitrarily small and $f_\infty$ is Lipschitz with the Lipschitz constant 1 as a limit of Lipschitz functions.
This completes the proof that  $\calF^{\leq M}$ is compact.


In the following, we will use the convention that $0 \log 0 = 0$.
For $f\in \calF^M$ let
\begin{align}\label{f5.1}
\beta(f) & = \sum_{k=1}^{\kf}\left( (t_k - t_{k-1}) \log (t_k - t_{k-1})
- \sum_{1\leq j \leq m_k} z^k_j(t_k - t_{k-1})
 \log \left( z^k_j(t_k - t_{k-1})\right)\right) \\
& =
-\sum_{k=1}^{\kf}\left( (t_k - t_{k-1}) 
 \sum_{1\leq j \leq m_k} z^k_j \log z^k_j\right). \label{f25.1}
\end{align}


\begin{lemma}\label{lem:beta_cont}
The function $\beta(f)$ is continuous on the set of floodings of $\calG$.
\end{lemma}

\begin{proof}
We order the family of all ``unit vectors'' $e^+$ and $e^-$, for all $e\in\calE$, in an arbitrary way. There are $2|\calE|$ elements in this family. We can change the sums $\sum_{1\leq j \leq m_k}$ to $\sum_{1\leq j \leq 2|\calE|}$ in \eqref{f5.1}--\eqref{f25.1}, with the understanding that for the extra $j$ in the sum, $z^k_j = 0$, and, therefore, the corresponding term in the sum is equal to 0. The point of the new convention is that the range of the summation index does not depend on the flooding $f$.

Consider two floodings $f$ and $\wt f$.
The objects related to $\wt f$ will be denoted $\wt t_k, \wt z^k_j$, etc.
Let $s_0=0< s_1 < s_2 <\dots < s_n < \infty$ be the ordering of all elements of $\{t_0, t_1, \dots , t_\kf\} \cup \{\wt t_0,\wt t_1, \dots ,\wt  t_{{\bf k}(\wt f)}\} $. If $(s_{i-1}, s_i) \subset (t_{k-1}, t_k)$ then we let $y^i_j = z^k_j$. If $(s_{i-1}, s_i) \subset (\wt t_{k-1},\wt t_k)$ then we let $\wt y^i_j =\wt z^k_j$. Then it follows from \eqref{f25.1} that
\begin{align*}
\beta(f) & = 
-\sum_{k=1}^{n}\left( (s_k - s_{k-1}) 
 \sum_{1\leq j \leq 2|\calE|} y^k_j \log y^k_j\right),  \\
\beta(\wt f) & = 
-\sum_{k=1}^{n}\left( (s_k - s_{k-1}) 
 \sum_{1\leq j \leq 2|\calE|} \wt y^k_j \log \wt y^k_j\right).
\end{align*}
In view of the  equality of \eqref{f5.1} and \eqref{f25.1},
\begin{align*}
\beta(f) & = \sum_{k=1}^{n}\left( (s_k - s_{k-1}) \log (s_k - s_{k-1})
- \sum_{1\leq j \leq 2|\calE|} y^k_j(s_k - s_{k-1})
 \log \left( y^k_j(s_k - s_{k-1})\right)\right), \\
\beta(\wt f) & = \sum_{k=1}^{n}\left( (s_k - s_{k-1}) \log (s_k - s_{k-1})
- \sum_{1\leq j \leq 2|\calE|} \wt y^k_j(s_k - s_{k-1})
 \log \left( \wt y^k_j(s_k - s_{k-1})\right)\right).
\end{align*}
Hence,
\begin{align}\label{f25.3}
&\beta(f)-\beta(\wt f) \\
& = -\sum_{k=1}^{n}
\sum_{1\leq j \leq 2|\calE|} \left( y^k_j(s_k - s_{k-1})
 \log \left( y^k_j(s_k - s_{k-1})\right)
- \wt y^k_j(s_k - s_{k-1})
 \log \left( \wt y^k_j(s_k - s_{k-1})\right)\right).\nonumber
\end{align}

Since $\kf \leq \n_\calG$ for every $f$, we have $n \leq 2 \n_\calG$.
Recall that $y^k_j \leq 1$ and $s_k - s_{k-1} \leq \zeta$. Hence, $y^k_j(s_k - s_{k-1}) \leq \zeta$ and $\wt y^k_j(s_k - s_{k-1}) \leq \zeta$.

Consider an arbitrarily small $\eps>0$.
Let $\eps_1>0$ be such that $4 \n_\calG |\calE| \eps_1 = \eps/2$.
Let $\delta>0$ be so small that if $b_1,b_2\in [0,\zeta]$ and $|b_1-b_2| < 20\delta$ then $|b_1 \log b_1 - b_2 \log b_2| \leq \eps_1$.

Now suppose that $\dist(f, \wt f) < \delta$. 
Consider an interval $[s_{k-1}, s_k]$ and the $j$-th outer normal vector.
Let $\Delta = y^k_j(s_k - s_{k-1})$ and $\wt \Delta =\wt y^k_j(s_k - s_{k-1})$. Without loss of generality, suppose that the $j$-th vector is $e^+$, corresponding to an edge $e$.
We will assume that the floodings $f$ and $\wt f$ do not have sources in $e$. The  case
when there are sources in $e$  is more complicated but does not pose any new
conceptual challenges; we leave it to the reader.

We will write $(e, (u_1, u_2])=\bigcup_{u\in(u_1,u_2]} \{(e,u)\}$ and similarly for other intervals (open or closed).
For some $x_1, x_2, x_3,\wt x_1,\wt x_2, \wt x_3\in[0,\ell(e)]$ such that $x_3 \leq x_2$ and $\wt x_3 \leq \wt x_2$, we have
\begin{align*}
(e,(0, \ell(e))) \cap  f(s_{k-1})
&= (e, (0,x_1]) \cup (e, [x_2,\ell(e))),\\
(e,(0, \ell(e))) \cap (f(s_k) \setminus f(s_{k-1}))
&= (e, (x_1, x_1+\Delta]) 
\cup (e, [x_3,x_2)),\\
(e,(0, \ell(e))) \cap \wt f(s_{k-1})
&= (e, (0,\wt x_1]) \cup (e, [\wt x_2,\ell(e))),\\
(e,(0, \ell(e))) \cap (f(s_k) \setminus f(s_{k-1}))
&= (e, (\wt x_1, \wt x_1+\wt\Delta]) 
\cup (e, [\wt x_3,\wt x_2)).
\end{align*}
Note that some, even all, of the eight line segments on the right hand side can be empty. 

Suppose that $|x_1 - \wt x_1| \geq 10 \delta$. Consider the case  $x_1 - \wt x_1 \geq 10 \delta$. Since $\dist(f, \wt f) < \delta$, we must have $\wt x_2 \leq x_1 - 8 \delta$. 
Let $t$ be the largest time such that for some 
$x_4, x_5, \wt x_4\in[0,\ell(e)]$,
\begin{align*}
(e,(0, \ell(e))) \cap  f(t)
&= (e, (0,x_4]) \cup (e, [x_5,\ell(e))),\\
(e,(0, \ell(e))) \cap \wt f(t)
&= (e, (0,\wt x_4]) \cup (e, [x_4,\ell(e))).
\end{align*}
Note that $t< s_{k-1}$, $x_5 \geq x_2$, $\wt x_4 \leq \wt x_1$ and $\wt x_2 \leq x_4 \leq x_1$.
Since $\wt x_2 \leq x_1 - 8 \delta$,
we have $x_4 \geq \wt x_2 + 4\delta$ or $x_4 \leq x_1 -4\delta$.

Suppose that $x_4 \geq \wt x_2 + 4\delta$ and let $x_6 =(\wt x_2 + x_4)/2$. Then $(e, x_6) \in f(t)$ but $(e, (x_6-2\delta, x_6+2\delta)) \cap \wt f(t) =\emptyset$. This contradicts the assumption that   $\dist(f, \wt f) < \delta$.

Consider the other case, namely, $x_4 \leq x_1 -4\delta$ and let $x_7 =(x_4+x_1)/2$. Then $(e, x_7) \in \wt f(t)$ but $(e, (x_7-2\delta, x_7+2\delta)) \cap  f(t) =\emptyset$. This contradicts the assumption that   $\dist(f, \wt f) < \delta$.

We conclude that $x_1 - \wt x_1 \geq 10 \delta$ cannot hold. The case $\wt x_1 -  x_1 \geq 10 \delta$ can be eliminated in a similar manner. This implies that $|x_1 - \wt x_1| < 10 \delta$. An analogous argument shows that
$|(x_1 +\Delta)- (\wt x_1+\wt \Delta)| < 10 \delta$. Recalling the definitions of $\Delta$ and $\wt \Delta$, we obtain
\begin{align*}
\left|y^k_j(s_k - s_{k-1})-\wt y^k_j(s_k - s_{k-1})\right|
= |\Delta - \wt \Delta|\leq 20 \delta.
\end{align*}
It follows that 
\begin{align*}
 &\left|  y^k_j(s_k - s_{k-1})
 \log \left( y^k_j(s_k - s_{k-1})\right)
- \wt y^k_j(s_k - s_{k-1})
 \log \left( \wt y^k_j(s_k - s_{k-1})\right)\right|
\leq \eps_1,
\end{align*}
and, therefore,
\begin{align*}
 &\left|-\sum_{k=1}^{n}
\sum_{1\leq j \leq 2 |\calE|} \left( y^k_j(s_k - s_{k-1})
 \log \left( y^k_j(s_k - s_{k-1})\right)
- \wt y^k_j(s_k - s_{k-1})
 \log \left( \wt y^k_j(s_k - s_{k-1})\right)\right)\right|\\
& \leq n \cdot 2|\calE| \cdot \eps_1
\leq 4 \n_\calG |\calE| \eps_1 = \eps/2.
\end{align*}
This and \eqref{f25.3} imply that $|\beta(f)-\beta(\wt f)| \leq \eps/2$.
\end{proof}


Let
\begin{align*}
\beta^* &= \beta^*(S,M) = \sup\left\{ \beta(f): f \in \calF^{\leq M} \right\},  \\
\calF^M_* &= \left\{f\in \calF^{\leq M}: \beta(f)= \beta^*(S,M)\right\}. 
\end{align*}

Since the set $\calF^{\leq M}$ is compact  and $\beta$ is continuous, $\beta$ attains  its supremum. Simple examples show that $\calF^M_*$ does not have to be a singleton, for example, when $S$ is a circle and $M$ is any positive integer. Theorem \ref{f22.1} (iii) implies that 
$ \beta^*(S,M) = \sup\left\{ \beta(f): f \in \calF^M\right\}$
and $\calF^M_* = \left\{f\in  \calF^M: \beta(f)= \beta^*(S,M)\right\}$ but we cannot take these statements as the definitions of $\beta^*$ and $\calF^M_*$ for technical reasons.

Recall the definition of a subdivision graph from Section \ref{f24.1}.
Consider a graph $\calG=(\calV,\calE)$ and a subdivision graph $\calG_n=(\calV_n,\calE_n)$.
We will associate a flooding $f_L$ of $\calG$ (not $\calG_n$) with any (deterministic) labeling $L$ of $\calV_n$.
First of all, note that a flooding $f$ is totally determined by the finite family $\{f([0, t_k]), 0\leq k \leq \kf\}$. This is because $f$ is continuous and the functions
$x^k_j$ are assumed to be linear on $[t_{k-1},t_k]$ for every $j$ and $k$.

Suppose that $L$ is a labeling of $\calV_n$ with $M$ peaks. 
Let $R=L^{-1}$,   $N_n = |\calV_n|$,
and for $0\leq k \leq N_n-1$, let
\begin{align*}
C_k = R(\{N_n, N_n-1,\dots, N_n-k\}).
\end{align*}
  Let $k_0=0$ and let $0< k_1 <k_2 < \dots < k_\alpha$ be the list of all $k_i>0$ such that at least one of the following conditions holds:

(i) $R(N_n-k_i)$ is a vertex of $\calG$, or 

(ii) $C_{k_i}$
contains all vertices of $\calV_n$ that belong to an edge $e$ of $\calG$ but $C_{k_i-1}$ does not contain some vertex in $e$. 

(iii) The vertex in $C_{k_i} \setminus C_{k_i-1}$ is a peak.

Strictly speaking (i) implies (ii) but we listed the two conditions separately to highlight the heuristic difference between the two cases. An arm of the labeling  reaches a vertex of $\calG$ in case (i). Case (ii) represents the situation when  two arms of the labeling meet in the interior of an edge of $\calG$.

Note that $\alpha$ (i.e., the number of $k_i$'s greater than 0) depends on $\calG,n$ and $L$.

Let $t_{k_i}$ be the total length of all edges of the metric graph $\calG_n$ whose both endpoints belong to $C_{k_i}$.
We define the flooding $f_L$ by declaring that for $0\leq i \leq \alpha$, $f_L([0,t_{k_i}])$ is the 
union of the set of all peaks of $L$ and all edges of $\calG_n$ whose both endpoints belong to $C_{k_i}$.

Let $D_n^L(t) = \calV_n \cap f_L([0,t])$ and let $\ominus$ denote the symmetric set difference. Let $\P^n_M$ denote the distribution of  random labeling of $\calG_n$ conditioned on having exactly $M$ peaks.

\begin{theorem}\label{j22.1}
Consider a graph $\calG$ and an integer $M>0$. For any $\eps>0$ there
exists $n_1$ such that for $n\geq n_1$,
\begin{align}\label{m30.1}
\P_M^n(\beta(f_L) \leq \beta^*(S,M) -\eps) 
 < \exp(-n \eps + 2\log^2n),
\end{align}
and
\begin{align}\label{m31.3}
\P_M^n\left(  
\sup_{0\leq t < \zeta} \left| D_n^L(t) \ominus 
R(\{N, N-1, \dots,  \lceil N(1- t /\zeta)\rceil\})\right|  < \eps n
\right) > 1-\eps.
\end{align}
\end{theorem}

\begin{proof}

{\it Step 1}.
The set 
\begin{align}\label{m31.1}
C_{k_i} \setminus C_{k_{i-1}} = 
R(\{N_n-k_i, N_n-k_i +1, \dots, N_n- k_{i-1}-1\})
\end{align}
consists of a certain number, say, $m_{k_i}$, of distinct ``line segments,'' i.e., $m_{k_i}$  sets of consecutive vertices $(v(k_i,j,1), v(k_i,j,2), \dots,v(k_i,j,U^{k_i}_j))$, $1\leq j \leq m_{k_i}$, 
where $U^{k_i}_j$ denotes the cardinality of the sequence,
$v(k_i,j,1)$ is adjacent to $C_{k_{i-1}}$ and
$v(k_i,j,U^{k_i}_j)$ lies on the boundary of $C_{k_i}$.
It follows from the construction that each sequence $(v(k_i,j,1), v(k_i,j,2), \dots,v(k_i,j,U^{k_i}_j))$ lies on only one edge  of $S$. 
Two of these sequences, corresponding to different values of $(k_i,j)$, may have at most one vertex in common, and if they do then  the common vertex is  the endpoint for each one of the sequences. 

Let us condition the random labeling on having exactly $M$ peaks, on the values of $L(N- \ell)$ for $0\leq \ell \leq k_{i-1}$,
and on the value of the set in \eqref{m31.1}. The conditional distribution of
\begin{align}\label{m31.2}
\left\{\left(L(v(k_i,j,1)), L(v(k_i,j,2)), \dots,L(v(k_i,j,U^{k_i}_j))\right),
1\leq j \leq m_{k_i}\right\},
\end{align}
can be described as follows. First, we randomly (uniformly) divide the family of labels (integers) $N_n-k_i, N_n-k_i +1, \dots, N_n- k_{i-1}-1$, into disjoint subsets of sizes $U^{k_i}_j$, $1\leq j \leq m_{k_i}$. Then we order integers in each subset and then assign them to vertices as in \eqref{m31.2}.
The version of the law of large numbers for the multinomial distribution given in \eqref{m30.2} implies that the values in \eqref{m31.2} decrease in an approximately linear way. We can repeat the same reasoning for each $k_i$, $1\leq i \leq \alpha$. It is straightforward to translate our observations  into the statement in \eqref{m31.3}.

\medskip
{\it Step 2}.
We will find a crude upper bound for the  number of distinct families $\{(v(k_i,j,1),v(k_i,j,U^{k_i}_j)), 1\leq j \leq m_{k_i},1\leq i \leq \alpha\}$ associated with all (deterministic) floodings with $M$ sources. 
The set $\bigcup_{i=1}^\alpha (C_{k_i} \setminus C_{k_i-1})$ contains all vertices of $\calG$, all peak  locations and at most one vertex in the subdivision graph for each edge of $\calG$ (and no other vertices). Hence, $\alpha \leq k_*$, where $k_*$ is the  sum of the numbers of vertices and edges of $\calG$ plus the number of sources, i.e., $M$. 
Since $m_{k_i}$ is bounded by the number of edges of $\calG$, $m_{k_i} \leq k_*$ for all $1\leq i \leq \alpha$, for every (deterministic) labeling $L$.
A crude upper estimate for the number of vertices in $\calV_n$ is $2n\zeta$.
All these remarks imply that the number of distinct families $\{(v(k_i,j,1),v(k_i,j,U^{k_i}_j)), 1\leq j \leq m_{k_i},1\leq i \leq \alpha\}$ is bounded by $(2n\zeta)^{2k_*^2}$. 

\medskip
{\it Step 3}.
Recall the notation related to the construction of $f_L$
presented before the statement of the theorem.
We match indices of arms of $f_L$ with the indices of the sequences
constructed above so that $(v(k_i,j,1), v(k_i,j,2), \dots,v(k_i,j,U^{k_i}_j))$ corresponds to $z^{k_i}_j$. Then for all $k_i$ and $j$,
\begin{align}\label{m28.1}
-1< nz^{k_i}_j (t_{k_i} - t_{k_i-1}) -  U^{k_i}_j\leq 0.
\end{align}

Let 
\begin{align*}
\bL_n=\bL_n((v(k_i,j,1),v(k_i,j,U^{k_i}_j)), 1\leq j \leq m_{k_i}, 1\leq i \leq k_L)
\end{align*}
be the set of all deterministic labelings $L$ of $\calG_n$ which correspond to the family of pairs $(v(k_i,j,1),v(k_i,j,U^{k_i}_j))$, $1\leq j \leq m_{k_i}$, $1\leq i \leq \alpha$, in the manner described above. Let
$U^i =k_i - k_{i-1}= \sum_{ 1\leq j \leq m_{k_i}} U^{k_i}_j$.
We use \eqref{m24.1} and \eqref{m28.1} to see that
\begin{align}\label{a24.1}
\log &|\bL_n((v(k_i,j,1),v(k_i,j,U^{k_i}_j)), 
1\leq j \leq m_{k_i}, 1\leq i \leq \alpha)|\\
&=
\log\prod _{i=1}^{\alpha}
\calN( U^i - m_{k_i}; U^{k_i}_1-1, U^{k_i}_2-1, \dots , U^{k_i}_{m_{k_i}}-1) \nonumber \\
&= n\sum _{i=1}^{\alpha}
  (t_{k_i} - t_{k_i-1})  \log (t_{k_i} - t_{k_i-1})\nonumber  \\
& \quad-  n\sum _{i=1}^{\alpha}
  \sum_{1\leq j \leq m_{k_i}} 
z^{k_i}_j(t_{k_i} - t_{k_i-1})  \log (z^{k_i}_j(t_{k_i} - t_{k_i-1}))
 + O(\log n)\nonumber \\
& = n \beta(f_L) + O(\log n) .\nonumber 
\end{align}

Let $\bM_n$ be the number of deterministic labelings of $\calG_n$ with $M$ peaks. 
Recall from Step 2 that the number of distinct families $\{(v(k_i,j,1),v(k_i,j,U^{k_i}_j)), 1\leq j \leq m_{k_i},1\leq i \leq \alpha\}$ is bounded by $(2n\zeta)^{2k_*^2}$. 
It follows from \eqref{a24.1} that for $b>0$ and large $n$,
\begin{align}\label{ma29.1}
&\P_M^n(\beta(f_L) \leq b) 
= \P_M^n \left(n\beta(f_L) \leq nb \right)\\
&\leq
\frac1 {\bM_n}
\sum_{(v(k_i,j,1),v(k_i,j,U^{k_i}_j)),1\leq j \leq m_{k_i}, 1\leq i \leq \alpha}
|\bL_n|
\bone_{\{ \log|\bL_n| \leq nb + \log^2n\}}(\bL_n) \nonumber \\
&\leq
\frac1 {\bM_n}
(2n\zeta)^{2k_*^2}
\exp(nb + \log^2n).\nonumber
\end{align}

\medskip
{\it Step 4}.
Given a function $f\in \calF^M_*$ we will construct a family $\calL_f^n$ of associated labelings of $\calG_n$. Let $B_k = \calV_n \cap f([0, t_k])$ for $0 \leq k \leq \kf$. Note that the index $k$ in $B_k$ has a different meaning than in  $C_k$. The number of vertices in $C_k$ is $k+1$ while the number of vertices in $B_k$ is typically of order $n$. The set $B_0$ may be empty. If it is non-empty then it contains at most $M$ vertices.

We will say that a labeling $L$ belongs to $\calL_f^n$ if it has exactly $M$ peaks and  $B_k \setminus B_{k-1} = R(\{i, i+1, \dots, i+\ell\})$ for some $i$ and $\ell$, for every $1\leq k \leq \kf$.

For $1\leq k \leq \kf$, the set $B_k \setminus B_{k-1}$
consists of a certain number, say, $m_k$, of distinct ``line segments,'' i.e., $m_k$  sets of consecutive vertices $(v(k,j,1), v(k,j,2), \dots,v(k,j,V^{k}_j))$, $1\leq j \leq m_{k}$, 
where $V^{k}_j$ denotes the cardinality of the sequence,
$v(k,j,1)$ is adjacent to $B_{k-1}$ and
$v(k,j,V^{k}_j)$ lies on the boundary of $B_{k}$.
It follows from the construction that each sequence $(v(k,j,1), v(k,j,2), \dots,v(k,j,V^{k}_j))$ lies on only one edge  of $S$. 
Two  sequences $(v(k,j,1), v(k,j,2), \dots,v(k,j,V^{k}_j))$ corresponding to different values of $(k,j)$ 
 may have at most one vertex in common, and if they do then  the common vertex is  the endpoint for each one of the sequences.  

We match indices of arms of $f$ with the indices of the sequences
constructed above so that $(v(k,j,1), v(k,j,2), \dots,v(k,j,V^{k}_j))$ corresponds to $z^{k}_j$. Then for all $k$ and $j$,
\begin{align}\label{ma29.2}
-1< nz^{k}_j (t_{k} - t_{k-1}) -  V^{k}_j\leq 0.
\end{align}

Let 
\begin{align*}
\bL_n^f=\bL^f_n((v(k,j,1),v(k,j,V^{k}_j)), 1\leq j \leq m_{k}, 1\leq k \leq \kf)
\end{align*}
be the set of all deterministic labelings $L$ of $\calG_n$ which correspond to the family of pairs $(v(k,j,1),v(k,j,V^{k_i}_j))$, $1\leq j \leq m_{k_i}$, $1\leq k \leq \kf$, in the manner described above.
Recall that $\bM_n$ is the number of deterministic labelings of $\calG_n$ with $M$ peaks.  Let
$V^k = \sum_{ 1\leq j \leq m_{k}} V^{k}_j$.
We use \eqref{m24.1}, \eqref{ma29.2} and the assumption that  $f\in \calF^M_*$ to see that
\begin{align*}
\log \bM_n
&\geq
\log |\bL^f_n((v(k,j,1),v(k,j,V^{k}_j)), 
1\leq j \leq m_{k}, 1\leq k \leq \kf)|\\
&=
\log\prod _{k=1}^{\kf}
\calN( V^k - m_{k}; V^{k}_1-1, V^{k}_2-1, \dots , V^{k}_{m_{k}}-1)\\
&= n\sum _{k=1}^{\kf}
  (t_{k} - t_{k-1})  \log (t_{k} - t_{k-1}) \\
& \quad-  n\sum _{k=1}^{\kf}
  \sum_{1\leq k \leq m_{k}} 
z^{k}_j(t_{k} - t_{k-1})  \log (z^{k}_j(t_{k} - t_{k-1}))
 + O(\log n)\\
& = n \beta(f) + O(\log n)
= n \beta^*(S,M) + O(\log n) .
\end{align*}
This and \eqref{ma29.1} imply that for any $\eps>0$ there
exist $c_1$ and $n_1$ such that for $n\geq n_1$,
\begin{align*}
\P_M^n&(\beta(f_L) \leq \beta^*(S,M) -\eps) 
\leq
\frac1 {\bM_n}
(2n\zeta)^{2k_*^2}
\exp(n(\beta^*(S,M) -\eps) + \log^2n)\\
&\leq c_1 n \exp( - n\beta^*(S,M))
(2n\zeta)^{2k_*^2}
\exp(n(\beta^*(S,M) -\eps) + \log^2n)\\
&= c_1 n 
(2n\zeta)^{2k_*^2}
\exp(-n \eps + \log^2n) < \exp(-n \eps + 2\log^2n).
\end{align*}
This proves \eqref{m30.1}.
\end{proof}

\section{Properties of floodings}\label{flood}


It is elementary to see that the set $\calF^M$ can be parametrized so that it becomes a subset of a finite dimensional space. By Lemma \ref{lem:beta_cont} the function $\beta$ defined in \eqref{f5.1} is  continuous over $\calF^M$, however, finding the maximizing $f$ in an effective way does not seem to be easy. 

Suppose that $\calH$ is obtained from $\calG$ by subdividing it with a finite number of vertices. It is not hard to see that the metric graph for $\calH$ can be embedded isometrically in the metric graph for $\calG$. That is, both graphs are represented by the same metric graph $S$, except with different edge labels and the set of vertices  of $\calG$ is a subset of the set of vertices of $\calH$. 

\begin{lemma}\label{f19.2}
For every $M$, $\sup\{\beta(f): f\in \calF^M(\calH)\}= \sup\{\beta(f): f\in \calF^M(\calG)\}$.
\end{lemma}

\begin{proof}
Note that $\calF^M(\calG) \subset \calF^M(\calH)$ so
\begin{align}\label{f19.1}
\sup\{\beta(f): f\in \calF^M(\calH)\}
\geq \sup\{\beta(f): f\in \calF^M(\calG)\}.
\end{align}

We will not analyze $\calF^M$ directly. Instead, we will use 
Theorem \ref{j22.1}. Note that adding a finite number of extra vertices to $\calG$ does not change the possible random labelings of the subdivision graphs---they are the same for subdivisions of $\calG$ and $\calH$ (this is almost true; see the note at the end of the proof). If we suppose that the inequality in \eqref{f19.1} is strict then  $\calF^M_*(\calH) \ne \calF^M_*(\calG)$ and distance, in the supremum norm metric, between these families of functions is strictly positive because $\beta$ is a continuous function on $\calF^M(\calH)$. Hence, the shape of a typical random labeling $L^M_n$, in the sense of \eqref{m31.3}, must be different for $\calH$ and $\calG$, for large $n$. This contradicts the fact that we are dealing with the same family of random labelings in both cases.

We address a subtle but minor problem related to the claim that adding a finite number of extra vertices to $\calG$ does not change the possible random labelings of the subdivision graphs, i.e., they are the same for subdivisions of $\calG$ and $\calH$. If an edge of length $r_e$ is divided by an extra vertex into edges with lengths $r'_e$ and $r''_e$ then 
it is not necessarily true that $\lceil n r_e\rceil=\lceil n r'_e\rceil + \lceil n r''_e\rceil$. Hence, there may be an extra  vertex in the subdivision graph for $\calH$. 
The effect on calculations in the proof of Theorem \ref{j22.1} is negligible for large $n$.
\end{proof}

Parts (i) and (ii) of the following definition are taken from \cite{twinp}.

\begin{definition}\label{j24.2}
(i) Suppose that a graph $\calG$ is a tree with $N$ vertices. We will say that $x$ is a centroid of $\calG$ if each subtree of $\calG$ which does not contain $x$ has at most $N/2$ vertices.

(ii) 
Suppose that $\calG$ is a tree and $v_1$ is one of its vertices.
 Define a partial order ``$\leq$'' on the set of vertices $\calV$ by declaring that for $v_2,v_3\in \calV$, we have $v_2\leq v_3$ if and only if $v_2$ lies on the geodesic between $v_1$ and $v_3$.
If $v_2\leq v_3$ then we will say that $v_2$ is an ancestor of $v_3$ and $v_3$ is a descendant of $v_2$. Let $\n_{v_1}(v_2)$ be the number of  descendants of $v_2$, including $v_2$.

(iii) Suppose that a metric graph $S$ represents a tree graph and the total length of all edges of $S$ is $\zeta$. 
We will say that $v\in S$ is the centroid of $S$ if for every connected subset $A$ of $S\setminus \{v\}$, the total length of edges in $A$ is less than or equal to $\zeta/2$.

\end{definition}

\begin{lemma}\label{m13.5}
(i) A finite  tree graph has at least one and at most two centroids. If it has two centroids then they are adjacent.

(ii) A metric graph $S$ representing a tree graph has exactly one centroid.

\end{lemma}

\begin{proof}
(i) This part appeared as Lemma 6.2 in \cite{twinp}.

(ii) One can discretize the metric tree $S$ by placing $n r_e$ equidistant vertices on $e$, for every edge $e$ of $S$. It is elementary to check that centroids of these graphs converge, as $n\to \infty$, to the unique centroid of $S$.
We leave the details to the reader. 
\end{proof}

Recall the notation from Definition \ref{m5.1}.


\begin{theorem}\label{m6.1}

Assume  that $\calG$ is a tree and and $f \in \calF^1_*$.

(i) Then $f(0)$ is the centroid of $S$.

(ii) Suppose that  $t\in (t_{k-1}, t_k)$ for some $1\leq k \leq \kf$. 
Let
$S^k_j(t)$ be the component of $S\setminus \{x^k_j(t)\}$ which does not contain $f(0)$. Let
$\zeta^k_j(t)$ be the total length of edges in $S^k_j(t)$. Then 
\begin{align}\label{m10.1}
z^k_j = \dot x^k_j(t) = \frac{\zeta^k_j(t)}{\sum_{j=1}^{m_k} \zeta^k_j(t)}
= \frac{\zeta^k_j(t)}{\zeta -t}.
\end{align}

\end{theorem}

The first and last equalities in \eqref{m10.1} follow directly from the definitions.

\begin{proof}[Proof of Theorem \ref{m6.1}]
(i) Consider a random labeling of a subdivision graph $\calG_n$ conditioned on having exactly one peak. We will estimate the distance of the peak from the centroid.

For a vertex $v$ in the subdivision graph $\calG_n$, let $p_v$ be the probability that  there exists only one peak and it is located at $v$. Recall $\n_v(y)$ from Definition \ref{j24.2}. If $v$ and $y$ are adjacent then, by \cite[(6.2)]{twinp},
\begin{align*}
\frac{p_v}{p_y} = \frac {\n_y(v )} { \n_v(y)}.
\end{align*}
Suppose that $v_0,v_1, \dots,v_k$ is a sequence of distinct vertices such that $v_{j-1} \lra v_j$ for all $j$ and $v_0$ is one of the centroids.
Let $N$ be the number of vertices of $\calG_n$ and note that $k\leq N/2$ because $v_0$ is a centroid. Then
$\n_{v_{j-1}}(v_j) \leq N/2 -j+1$ and $\n_{v_{j}}(v_{j-1}) \geq N/2 +j-1$
for $j=1,\dots,k$, so
\begin{align*}
\frac{p_{v_k}}{p_{v_0}}
&= \prod_{j=1}^k  \frac {p_{v_j}}{p_{v_{j-1}}}
= \prod_{j=1}^k  \frac {\n_{v_{j-1}}(v_j )} { \n_{v_j}(v_{j-1})}
\leq \prod_{j=1}^k  \frac {N/2 -j+1} { N/2 +j-1}
= \prod_{j=1}^k  \frac {1 -2(j-1)/N} { 1 +2(j-1)/N}\\
&= \exp\left( \sum_{j=1}^k \left(\log(1 -2(j-1)/N) - \log( 1 +2(j-1)/N)\right) \right)\\
&\leq \exp\left( \sum_{j=1}^k \log(1 -2(j-1)/N)  \right)
\leq \exp\left( -\sum_{j=1}^k 2(j-1)/N  \right)\\
&= \exp( (k-k^2)/N).
\end{align*}
Let $\calA_k$ be the set of all vertices at distance $k$ from one of the centroids and let $\calC_j= \bigcup_{k\geq j} \calA_k$. Since we are dealing with a subdivision graph, the cardinality of $\calA_k$ is bounded by some $K<\infty$ for all $N$ and $k$. 
If $k\geq 2$ then $k^2 - k \geq k^2/2$. Hence
for $j \geq 2$, 
\begin{align*}
 \sum_{v\in \calC_j} p_v
&\leq \sum_{v\in \calC_j} \frac{p_v}{p_{v_0}}
= \sum_{k\geq j} \sum_{v\in \calA_k} \frac{p_v}{p_{v_0}}
\leq \sum_{k\geq j} \sum_{v\in \calA_k} \exp( (k-k^2)/N)
\leq \sum_{k\geq j} K \exp( (k-k^2)/N)\\
&\leq \sum_{k\geq j} K \exp( -k^2/(2N))
\leq K \int_{j-1}^\infty \exp(-u^2/(2N)) du\\
&= K \sqrt{2N} \int_{(j-1)/\sqrt{2N}}^\infty \exp(-s^2) ds\\
&\leq K \sqrt{2N} \frac 1 { 2 (j-1)/\sqrt{2N}} \exp ( - ((j-1)/\sqrt{2N})^2)\\
&= \frac {KN}{j-1} \exp(- (j-1)^2 /(2N)).
\end{align*}
If we take $j = \lfloor N^{1/2+\alpha}\rfloor$ for some $\alpha >0$ and let $N\to \infty $ then 
the  quantity on the last displayed line goes to 0. 
Recall that centroids of subdivision graphs $\calG_n$ converge to the centroid of $S$ when $n\to \infty$.
It follows that for any neighborhood $U$ of the centroid $v_*$ of $S$ and for any $c_1>0$, there exists $N_1$ such that  if $\calG_n$ is a subdivision graph with $|\calV_n| \geq N_1$ and we condition a random labeling of the subdivision graph $\calG_n$ to have only one peak then the probability that the peak is outside $U$ is less than $c_1$.
This and \eqref{m31.3} easily imply part (i) of the theorem.

(ii)
Let $f$ be the unique flooding such that (a) $f(0)$ is the centroid $v_*$ of $S$, (b) $f(t_k)$ contains a vertex of degree greater than or equal to 3 for all $1\leq k \leq \kf$, and (c) the condition \eqref{m10.1} is satisfied for all  $1\leq k \leq \kf$ and  $t\in (t_{k-1}, t_k)$.  We will show that random labelings of subdivision graphs conditioned on having a single peak converge to $f$ in the sense of Theorem \ref{j22.1}. In view of Theorem \ref{j22.1} and part (i) of the present theorem, this will prove that  $f \in \calF^1_*$ and $\calF^1_*$ contains only one function and, in turn, will imply part (ii) of the present theorem. It remains to show that random labelings of subdivision graphs conditioned on having a single peak converge to $f$.

Let $d_{k-1}$ be the number of connected components of $S\setminus f(t_{k-1})$.
Let $\zeta_j^{k-1}$ be the total length of all edges of the $j$-th connected component of $S\setminus f(t_{k-1})$, for $j=1,\dots, d_{k-1}$, and let $\zeta_- = 1\land \min_{1\leq k \leq \kf} \min_{j=1,\dots, d_{k-1}} \zeta^{k-1}_j$.

Recall $e^k_j$ from Definition \ref{m5.1}.
Consider arbitrarily small $c_1,c_2>0$ and suppose that
\begin{align}\label{a27.1}
\eps \in\left(0, \min_{1\leq k \leq \kf}
\min_{1\leq j \leq m_k} \frac{\zeta_-(t_k - t_{k-1})\ell(e^k_j)}{4\zeta\kf}\right).
\end{align}
Let $L_n$ denote a random labeling of $\calG_n$ and let $R_n = L_n^{-1}$ be its inverse function. 
Let $N_n = |\calV_n|$ and let  $C^n_j = R_n(\{N_n, N_n-1, \dots, N_n-j\})$  for $0 \leq j \leq N_n-1$. For $0\leq k \leq \kf-1$, let $B^n_k = C^n_j$ where $j = \lfloor N_n (t_k+(k+1)\eps)/\zeta \rfloor$.
Let
$D^n_k \subset S$ be the union of all geodesics in $S$ between points of $B^n_k$. 
For $0\leq k \leq \kf-1$, let
$\wh F^n_k $ be the event 
$\{  f(t_k) \subset D^n_k \subset f(t_k+2(k+1)\eps\zeta/\zeta_-) \} $.
Condition \eqref{a27.1} implies that $\prt f(t_{k+1}) \cap f(t_k+2(k+1)\eps\zeta/\zeta_-) = \emptyset$ for all $k$.
For $1\leq k \leq \kf$, let
$\wt F^n_k $ denote the event that  $L_n/n$ restricted to $\calV_n \cap e^k_j$ is within $c_1$ of a linear function, in the supremum norm, for every $1\leq j \leq m_k$.
Let $F^n_0 =\wh F^n_0$, $F^n_\kf =\wt F^n_\kf $, and
for $1\leq k \leq \kf-1$, let
$F^n_k = \wh F^n_k \cap \wt F^n_k$.
 We will prove that there exists $n_1$ such that for $n\geq n_1$, the following statements are true,
\begin{align}\label{m10.2}
\P_1(F^n_0) &\geq 1-c_2, \\
\P_1(F^n_k \mid F^n_{k-1}) &\geq 1-c_2, \qquad 1 \leq k \leq \kf. 
\label{m10.3}
\end{align}

By the last paragraph of part (i) of the proof,  for large $n$,
\begin{align}\label{m11.2}
\P_1\left(\rho( R_n(N_n) , v_*) < \eps \zeta_-/(8\zeta)\right) > 1-c_2/2.
\end{align}

Consider $v_0 \in \calV_n$ such that $\rho(v_0, v_*) < \eps \zeta_-/(8\zeta)$. 
If $v_0 = v_*$ and $\{R_n(N_n)=v_0\}$ then $F^n_0$ holds.

Consider the case when $v_0\ne v_*$.
We can make $\eps>0$ smaller, if necessary, an assume that $n$ is large so that the number  of connected components of  $\calV_n \setminus \{R_n(N_n)\}$ is 2.
Let $\calV_n^{0,1}$ and $\calV_n^{0,2}$, be the  connected components of  $\calV_n \setminus \{R_n(N_n)\}$.
Note that if $n$ is sufficiently large then
\begin{align}\label{m11.1}
|\calV_n^{0,j}| \geq (\zeta_-/(2\zeta)) |\calV_n|, \qquad j=1,2.
\end{align}

Given the event $\{R_n(N_n)=v_0\}$, the event that there are no other peaks is independent under $\P$ from the allocation of integers in $[N_n-1]$ to the sets $\calV_n^{0,1} $ and $\calV_n^{0,2}$, i.e., it is independent  from the $\sigma$-field generated by the random sets $L_n(\calV_n^{0,1})$ and $ L_n(\calV_n^{0,2})$. It follows that, under $\P_1$ conditioned by $\{R_n(N_n)=v_0\}$, the random sets $L_n(\calV_n^{0,1})$ and $L_n(\calV_n^{0,2})$ are distributed uniformly, i.e., all allowed pairs of sets are equally likely.
This and \eqref{m11.1} imply that if  $n$ is sufficiently large then
\begin{align}
\P_1&\left(\, \left|\,[\,\lfloor N_n(1-\eps/\zeta)\rfloor, N_n] \cap L_n(\calV_n^{0,j})\right|
\geq N_n \eps \zeta_-/(4\zeta^2)
\text{  for  } j=1,2 \mid R_n(N_n)=v_0\right) \nonumber \\
& > 1-c_2/4.\label{m11.3}
\end{align}
If the events in \eqref{m11.2} and \eqref{m11.3} hold and the labeling has only one peak then $L_n(v_*)$ must be one of the top $N_n \eps /\zeta$ values of $[N_n]$ and, therefore, $\{  f(0)=f(t_0) \subset D^n_0 \} $ holds. 

Condition \eqref{m10.1} implies that $f(2\eps\zeta/\zeta_-)$ extends at least $2\eps$ units from $v_*$ along each  connected component of $S\setminus \{v_*\}$.
We  have assumed that
$\rho(v_0, v_*) < \eps \zeta_-/(8\zeta)$. 
These observations and the fact that $|B^n_0| = \lceil (\eps/\zeta) N_n\rceil$ imply  that the event
$\{  D^n_0 \subset f(2\eps\zeta/\zeta_-) = f(t_0+2(0+1)\eps\zeta/\zeta_-) \} $ holds.

The last claim and the estimates  \eqref{m11.2} and \eqref{m11.3}
show that the $\P_1$-probability of the event $F^n_0 = \{f(t_0) \subset D^n_0 \subset  f(t_0+2(0+1)\eps\zeta/\zeta_-) \} $ is greater than $1-3c_2/4$ so \eqref{m10.2} is proved.

\medskip

Next we will prove \eqref{m10.3}.
Suppose that the event $ F^n_{k-1}$ holds for some $ 1 \leq k \leq \kf$.

The assumption that $\wh F^n_{k-1}$ holds and \eqref{a27.1} imply that  the number  of connected components of  $\calV_n \setminus B^n_{k-1}$ is $d_{k-1}$.
Let $\calV_n^{k-1,1},\calV_n^{k-1,2}, \dots ,\calV_n^{k-1,d_{k-1}} $ be the  connected components of  $\calV_n \setminus B^n_{k-1}$.

Given $B^n_k$ and the event that there is only one peak in $B^n_k$,  the event that there are no other peaks is independent under $\P$ from the allocation of integers in $[N_n]\setminus L(B^n_k)$ to the sets $\calV_n^{k-1,1},\calV_n^{k-1,2}, \dots ,\calV_n^{k-1,d_{k-1}} $, i.e., it is independent  from the $\sigma$-field generated by the random sets $L_n(\calV_n^{k-1,1}), \dots ,L_n(\calV_n^{k-1,d_{k-1}}) $. It follows that, under $\P_1$ conditioned on $B^n_k$ and one peak in $B^n_k$,  all allowed  sets $L_n(\calV_n^{k-1,1}), \dots ,L_n(\calV_n^{k-1,d_{k-1}}) $ are equally likely.
This  implies that if  $n$ is sufficiently large then for $j=1,\dots, d_{k-1}$,
\begin{align}\label{m12.3}
\P_1&\Bigg(\, \left|\left[
\left\lfloor N_n(1-(t_k+(k+1)\eps)/\zeta ) \right\rfloor ,
\left\lceil N_n(1-(t_{k-1}+k\eps)/\zeta) \right\rceil \right]
\cap L_n(\calV_n^{k-1,j})\right|\\
&\qquad\in \left( 
\frac{N_n ((t_k-t_{k-1}) +\eps/2) \zeta^{k-1}_j}{\zeta(\zeta - t_{k-1})},
\frac{N_n ((t_k-t_{k-1}) +3\eps/2) \zeta^{k-1}_j}{\zeta(\zeta - t_{k-1})}
\right)
 \mid F^n_{k-1}\Bigg) \nonumber \\
& \geq 1-c_2/(2d_{k-1}).\nonumber
\end{align}
We will explain the geometric meaning of the event in the above formula. If the labeling has a single peak then label values are monotone along paths emanating from $R_n(N_n)$. Assuming that this is the case, the set
\begin{align*}
\left[
\left\lfloor N_n(1-(t_k+(k+1)\eps)/\zeta ) \right\rfloor ,
\left\lceil N_n(1-(t_{k-1}+k\eps)/\zeta) \right\rceil \right]
\cap L_n(\calV_n^{k-1,j})
\end{align*}
represents
is a sequence of adjacent vertices. An estimate of the graph distance between the extreme vertices in this sequence is given in \eqref{m12.3}. Ignoring $\eps$-size corrections, this distance is 
\begin{align*}
\frac{N_n (t_k-t_{k-1}) \zeta^{k-1}_j}{\zeta(\zeta - t_{k-1})}.
\end{align*}
This is equivalent, up to negligible terms, to the $S$-distance
\begin{align*}
\frac{ (t_k-t_{k-1}) \zeta^{k-1}_j}{\zeta - t_{k-1}}
\end{align*}
between the extreme vertices in the sequence.
The last quantity is the distance between $\prt f(t_{k-1})$ and $\prt f(t_k)$ along the $j$-th component of $S\setminus f(t_{k-1})$
because we have assumed that $f$ satisfies \eqref{m10.1}.

If the event in  \eqref{m12.3} and $\{  f(t_{k-1}) \subset D^n_{k-1} \} $
 hold and there is only one peak in $B^n_k$ then, for every $j=1,\dots, d_{k-1}$, then the vertex in $\calV_n^{k,j}$ corresponding to $\prt f(t_k)$ belongs to $B^n_k$ and, therefore, 
$\{  f(t_k) \subset D^n_k \}$.

Note that $ \zeta^{k-1}_j/(\zeta - t_{k-1}) \leq 1$. It follows that if the event in  \eqref{m12.3}  and
 $\{ D^n_{k-1} \subset f(t_{k-1}+2k\eps\zeta/\zeta_-) \} $ hold then
$ D^n_{k} \subset f(t_{k}+2(k+1)\eps\zeta/\zeta_-)  $. 

The above argument proves the following weaker version of \eqref{m10.3},
\begin{align}\label{m13.1}
\P_1(\wh F^n_k \mid F^n_{k-1}) &\geq 1-c_2/2, \qquad 1 \leq k \leq \kf.
\end{align}

The event $\wt F^n_k$ is concerned with linearity of the function $L_n$.
A standard argument, similar to that that gives
\eqref{m30.2}, shows that for large $n$,
\begin{align*}
\P_1(\wt F^n_k \mid F^n_{k-1}) &\geq 1-c_2/2, \qquad 1 \leq k \leq \kf.
\end{align*}
This and \eqref{m13.1} prove \eqref{m10.3}.

Since $c_1$ and $c_2$ are arbitrarily small, it follows from \eqref{m10.2}-\eqref{m10.3} that the $\P_1$-probability of $ \bigcap_{0 \leq k \leq \kf} F^n_k$ goes to 1 as $n\to \infty$. This implies that  random labelings of subdivision graphs conditioned on having a single peak converge to $f$ in the sense of Theorem \ref{j22.1}. This completes the proof.
\end{proof}

\begin{lemma}\label{m21.3}
Suppose that $M\geq 2$, $f \in \calF^M_*$, $v\in S$ and $v\notin f(t)$ for all $t< \zeta$. Assume that $x^\kf_i(t)$ and $x^\kf_j(t)$ approach $v$ as $t\to \zeta$. Then $z^\kf_i = z^\kf_j$.
\end{lemma}

\begin{proof}
Let $k = \kf -1$.  If $v$ lies half way between $ x^\kf_i(t_k)$ and $ x^\kf_j(t_k)$ then we must have $z^\kf_i = z^\kf_j$.

Suppose that $v$ is not the midpoint of the geodesic between $ x^\kf_i(t_k)$ and $ x^\kf_j(t_k)$. 
Let $\ol z = (z^\kf_i + z^\kf_j)/2$.
Consider a flooding $g$ constructed by modifying $f$ as follows. The function 
$x^\kf_i(t)$ is replaced with a function $\wh x^\kf_i(t)$, starting at $\wh x^\kf_i(t_k)=x^\kf_i(t_k)$,  moving along the same edge and in the same direction as $x^\kf_i(t)$ with the 
speed $\ol z$. We define $\wh x^\kf_j(t)$ in the analogous way.
All other functions $x^n_r$ associated to $f$ remain the same for $g$.
Then, by the properties of entropy,
\begin{align*}
\beta(g) - \beta(f) &=
 -(t_{k+1} - t_{k})  \left(z^\kf_i \log z^\kf_i + z^\kf_j \log z^\kf_j\right)
+ 2 (t_{k+1} - t_{k}) \ol z \log \ol z \\
&= -(t_{k+1} - t_{k})  \left(z^\kf_i \log z^\kf_i + z^\kf_j \log z^\kf_j 
-2  \ol z \log \ol z\right) > 0.
\end{align*}
This contradicts the assumption that $f \in \calF^M_*$ and thus proves the lemma.
\end{proof}

Recall the notation from Definition \ref{m5.1}.
We will say that $f$ has a dormant arm if for some $i,j$ and $k$ we have
 $x^k_j(t_{k-1}) = x^{k-1}_i(t_{k-2})$ and
 $x^k_j(t_{k-1})\ne x^k_j(t_k)$.

Next suppose that $M>1$. For every $v\in f(0)$ let $S_v(t)$ be the connected component of $f(t)$ which contains $v$, for $t\in[0, \zeta)$. Let $S_v$ be the closure of $\bigcup_{1\leq t < \zeta} S_v(t)$.

\begin{theorem}\label{f22.1}
If $M\geq 1$ and $f \in \calF^M_*$ then the following hold.

(i) The function $k \to m_k$ is non-decreasing. 

(ii) 
The function $f$ has no dormant arms.

(iii) For all $t < \zeta$, $f(t)$ does not contain any vertex of degree 1.

(iv) For all $t < \zeta$, $f(t)$ has exactly $M$ disjoint connected components.

(v) For all $t < \zeta$, $f(t)$ contains no loops.

(vi) For all $t < \zeta$, the points in the set $f(0)$ are centroids of the tree metric graphs comprising $f(t)$, i.e., connected components of $f(t)$.

(vii) 
Suppose that $y\in f(0)$ and let $S_y(t)$ be the connected component of $f(t)$ which contains $y$, for $t\in[0, \zeta)$. Let $S_y$ be the closure of $\bigcup_{1\leq t < \zeta} S_y(t)$.
Suppose that  $t\in (t_{k-1}, t_k)$ for some $1\leq k \leq \kf$, $v,w\in \prt f(t) \cap S_y$, $v=x^k_i(t)$, $w=x^k_j(t)$, and let $S_{yv}$ denote the connected component of $S_y\setminus \{v\}$ which does not contain $y$; let $S_{yw}$ be defined in the analogous way.  Let
$\zeta_v$ be the total length of edges in $S_{yv}$, and let $\zeta_w$ be the total length of edges in $S_{yw}$. Then,
\begin{align*}
\frac{z^k_j}{z^k_i} = \frac{\dot x^k_j(t)}{\dot x^k_i(t)} = \frac{\zeta_w}{\zeta_v}.
\end{align*}

(vii) 
Suppose that $y_1, y_2\in f(0)$, $t\in (t_{k-1}, t_k)$ for some $1\leq k \leq \kf$, $v\in \prt f(t) \cap S_{y_1}$, $v=x^k_i(t)$, $w\in \prt f(t) \cap S_{y_2}$, $w=x^k_j(t)$, and let $S_{y_1v}$ denote the connected component of $S_{y_1}\setminus \{v\}$ which does not contain $y_1$. Let $S_{y_2w}$ be the connected component of $S_{y_2}\setminus \{w\}$ which does not contain $y_2$.  Let
$\zeta_v$ be the total length of edges in $S_{y_1v}$, and let $\zeta_w$ be the total length of edges in $S_{y_2w}$. Then,
\begin{align}\label{m15.11}
\frac{z^k_j}{z^k_i} = \frac{\dot x^k_j(t)}{\dot x^k_i(t)} = \frac{\zeta_w}{\zeta_v}.
\end{align}

\end{theorem}

\begin{proof}

(i) Suppose that $f\in \calF^M$, $\kf \geq 2$, and $m_{k+1} < m_{k}$ for some $k+1\leq \kf$. 
We will construct a function $\ul f \in \calF^M$ with $\beta(\ul f) > \beta(f)$. Objects related to $\ul f$ will be denoted as in Definition \ref{m5.1} but underlined, for example, $\ul x^k_j$'s will play the role of $x^k_j$'s for $\ul f$.

The idea of the construction is the following. First,
suppose without loss of generality that the functions $x^k_j$ and $x^{k+1}_i$ are labeled in such a way that
\begin{align}\label{f19.6}
 \{x^{k+1}_1(t_{k}),
x^{k+1}_2(t_{k}), \dots, x^{k+1}_{m_{k+1}}(t_{k})\} \subset \{ x^k_2(  t_k), x^k_3(  t_k), \dots, x^k_{m_k}(  t_k)\}.
\end{align}

Suppose that $\eps \in (0, t_k - t_{k-1})$ is small and let $\ul t_k = t_k-\eps$.
We add an extra vertex to $\calG$ at $x^k_1(\ul t_k )$ so that we can change the values of $\ul z^m_j$'s at the time $\ul t_k$. We can add an extra vertex in view of Lemma \ref{f19.2}.

The function $\ul f$ will be such that $\ul f(t) = f(t)$ for $t\leq t_{k-1}$ and $t\geq t_{k+1}$. Moreover, ${\mathbf k}(\ul f) = \kf$ and $\ul t_j = t_j$ for $j\ne k$.
The number of active arms  of $\ul f$ on the interval $(\ul t_{k-1}, \ul t_k)$ will be the same as for $f$ on $( t_{k-1},  t_k)$ but $\ul f$ will have one more active arm on $(\ul t_{k}, \ul t_{k+1})$ than the number of active arms of $f$ on $( t_{k},  t_{k+1})$. 
In symbols, $\ul m_k = m_k $ and $ \ul m_{k+1} = m_{k+1} +1$.

We set
\begin{align*}
\ul x^k_j (\ul t_{k-1}) &=  x^k_j (t_{k-1}), \qquad j = 1, \dots, m_k, \\ 
\ul x^k_1 (\ul t_{k}) &=  x^k_1 (t_{k}-\eps),  \\ 
\ul x^k_j (\ul t_{k}) &=  x^k_j (t_{k}), \qquad j = 2, \dots, m_k, \\ 
\ul x^{k+1}_j (\ul t_{k}) &=  x^{k+1}_j (t_{k}), \qquad j = 1, \dots, m_{k+1}, \\ 
\ul x^{k+1}_j (\ul t_{k+1}) &=  x^{k+1}_j (t_{k+1}), \qquad j = 1, \dots, m_{k+1}, \\ 
\ul x^{k+1}_{\ul m_{k+1}} (\ul t_{k}) &=  x^k_1 (t_{k}-\eps), \\ 
\ul x^{k+1}_{\ul m_{k+1}} (\ul t_{k+1}) &=  x^k_1 (t_{k}).
\end{align*}
We removed  a short stretch of the range of $x^k_1$ from the $k$-th stage of $f$ and added this short interval as an extra arm to $\ul f$ at the $(k+1)$-st stage. 

In the following calculation, we will write $\ell = k+1$ so that we can reuse the formula later in the proof.
Let $a= 1 - \eps / (t_{k} - t_{k-1})$ so that $\eps = (t_{k} - t_{k-1})(1-a)$.
It is easy to check that
\begin{align}\label{f19.3}
\ul z^{k} _j &= z^{k} _j /a \quad \text{  for  } 2\leq j \leq m_{k},  \\
\ul z^{k} _1 &= z^{k} _1 , \nonumber \\
\ul z^{\ell} _j = \ul z^{k+1} _j &= z^{k+1} _j 
\frac{ t_{k+1} - t_{k} }{ t_{k+1} - t_{k} + \eps} 
=\frac{ z^{\ell}_j (t_{\ell} -  t_{\ell-1})}{t_{\ell} -  t_{\ell-1} + \eps}\quad \text{  for  } 1\leq j\leq \ul m_{k+1}-1, \nonumber \\
\ul z^{\ell} _{\ul m_{\ell}}
=\ul z^{k+1} _{\ul m_{k+1}}  &= z^k_1 \frac{\eps }{ t_{k+1} - t_{k} + \eps}
= z^k_1\frac{\eps}{t_{\ell} -  t_{\ell-1} + \eps}. \nonumber 
\end{align}

Since the functions $f$ and $\ul f$ agree outside $[t_{k-1}, t_\ell] = [t_{k-1}, t_k] \cup [t_{\ell-1}, t_{\ell}]$, we have
\begin{align}\label{f19.5}
\beta(\ul f) - \beta( f)&=
- (\ul t_{k} - \ul t_{k-1}) \sum_{j=1}^{\ul m_{k}} \ul z^{k}_j \log \ul z^{k}_j
 - (\ul t_{k+1} - \ul t_{k}) \sum_{j=1}^{\ul m_{k+1}} \ul z^{k+1}_j \log \ul z^{k+1}_j 
 \\
& \quad
+ ( t_{k} -  t_{k-1}) \sum_{j=1}^{m_{k}}  z^{k}_j \log  z^{k}_j
+( t_{k+1} -  t_{k}) \sum_{j=1}^{m_{k+1}}  z^{k+1}_j \log  z^{k+1}_j 
 \nonumber \\
 &=
- (\ul t_{k} - \ul t_{k-1}) \sum_{j=1}^{\ul m_{k}} \ul z^{k}_j \log \ul z^{k}_j
 - (\ul t_{\ell} - \ul t_{\ell-1}) \sum_{j=1}^{\ul m_{\ell}} \ul z^{\ell}_j \log \ul z^{\ell}_j 
\label{m5.7} \\
& \quad
+ ( t_{k} -  t_{k-1}) \sum_{j=1}^{m_{k}}  z^{k}_j \log  z^{k}_j
+( t_{\ell} -  t_{\ell-1}) \sum_{j=1}^{m_{\ell}}  z^{\ell}_j \log  z^{\ell}_j 
 \label{m5.8} \\
 &=  - ( t_{k} -  t_{k-1})a
z^{k} _1 \log(z^{k} _1 )  \nonumber \\
& \quad - ( t_{k} -  t_{k-1})a
\sum_{j=2}^{m_{k}}  (z^{k}_j/a) \log  (z^{k}_j /a)  \nonumber \\
& \quad
- ( t_{\ell} -  t_{\ell-1} + \eps) 
\sum_{j=1}^{m_{\ell}} \frac{ z^{\ell}_j (t_{\ell} -  t_{\ell-1})}{t_{\ell} -  t_{\ell-1} + \eps} \log  \frac{ z^{\ell}_j (t_{\ell} -  t_{\ell-1})}{t_{\ell} -  t_{\ell-1} + \eps} \nonumber \\
&\quad - ( t_{\ell} -  t_{\ell-1} + \eps) 
\frac{z^k_1\eps }{ t_{\ell} - t_{\ell-1} + \eps} \log \frac{z^k_1\eps }{ t_{\ell} - t_{\ell-1} + \eps} \nonumber \\
& \quad
+ ( t_{k} -  t_{k-1}) \sum_{j=1}^{m_{k}}  z^{k}_j \log  z^{k}_j
+( t_{\ell} -  t_{\ell-1}) \sum_{j=1}^{m_{\ell}}  z^{\ell}_j \log  z^{\ell}_j 
 \nonumber \\
 &=  - ( t_{k} -  t_{k-1})a
z^{k} _1 \log(z^{k} _1 ) \nonumber \\
& \quad + ( t_{k} -  t_{k-1}) \left(
z^{k}_1 \log  z^{k}_1 + \sum_{j=2}^{m_{k}}  z^{k}_j \log  a \right)  \nonumber \\
&\quad + ( t_\ell -  t_{\ell-1} ) 
\sum_{j=1}^{m_\ell} z^{\ell}_j \log \frac{ t_{\ell} -  t_{\ell-1}}{t_{\ell} -  t_{\ell-1} + \eps} \nonumber \\
&\quad - z^k_1\eps \log \frac{z^k_1 \eps }{ t_k - t_{k-1} + \eps} \nonumber \\
 &=   ( t_{k} -  t_{k-1})(1-a)
z^{k} _1 \log(z^{k} _1 ) \label{m5.2} \\
& \quad + ( t_{k} -  t_{k-1}) \sum_{j=2}^{m_{k}}  z^{k}_j \log  a  
\label{m5.3} \\
&\quad + ( t_\ell -  t_{\ell-1} ) 
\sum_{j=1}^{m_\ell} z^{\ell}_j \log \frac{ t_{\ell} -  t_{\ell-1}}{t_{\ell} -  t_{\ell-1} + \eps} \label{m5.4} \\
&\quad - z^k_1\eps \log \frac{z^k_1 \eps }{ t_k - t_{k-1} + \eps}.  \label{m5.5}
\end{align}
Recall that $f$ is given and fixed so all $z^k_j$'s, $t_k$'s, etc. are fixed. We  let  $\eps = (t_{k-1} - t_{k-2})(1-a)
\downarrow 0$ so that  $a\uparrow 1$.
Then the quantities on each of the lines \eqref{m5.2}-\eqref{m5.4} are of the order $O(\eps)$ for small $\eps > 0$. The  quantity in \eqref{m5.5} is  greater than $c \eps |\log \eps|$ for small $\eps >0$ so it dominates the other quantities. This shows that $\beta(\ul f) - \beta( f) >0$ for small $\eps >0$. Hence, $f\notin \calF^M_*$.

\medskip

(ii)
Suppose that $f\in \calF^M$ has a dormant arm, that is, there exist $i,j$ and $k$ such that $x^k_j(t_{k-1})\ne x^k_j(t_k)$ and
 $x^k_j(t_{k-1}) = x^{k-1}_i(t_{k-2})$. We assume without loss of generality that $j=1$ and set $\ell = k-1$, so that $x^k_1(t_{k-1})\ne x^k_1(t_k)$ and $x^k_1(t_{k-1}) = x^\ell_i(t_{\ell-1})$.

The strategy of the proof of part (ii) is the same as for part (i). We will construct a function $\ul f$ with $\beta(\ul f) > \beta( f) $ by shifting some growth from the $k$-th time interval to the $\ell$-th time interval.
Recall the notation conventions for objects related to $\ul f$ from part (i) of the proof.
Suppose that $\eps \in (0, t_k - t_{k-1})$, ${\mathbf k}(\ul f) = \kf$ and let
\begin{align*}
\ul t_\ell &= t_\ell + \eps  ,\\
\ul t_j &= t_j  \quad \text{  if  }j< \ell \text{  or  } j > k-1 .
\end{align*}
We set $\ul m_j = m_j $ for $j\ne \ell$ and $ \ul m_{\ell} = m_{\ell} +1$.
Let
\begin{align*}
\ul x^n_j (\ul t_{n-1}) =  x^n_j (t_{n-1})
\quad \text{and} \quad
\ul x^n_j (\ul t_{n}) =  x^n_j (t_{n}),
\end{align*}
for all $n=1, \dots, \kf$ and $ j = 1, \dots, m_n$, except that we let
$\ul x^k_1 (\ul t_{k-1}) =  x^k_1 (t_{k-1}+\eps)$.
We define
\begin{align*}
\ul x^{\ell}_{\ul m_{\ell}} (\ul t_{\ell-1}) =  x^k_1 (t_{k-1}), 
\quad \text{and} \quad 
\ul x^{\ell}_{\ul m_{\ell}} (\ul t_{\ell}) =  x^k_1 (t_{k-1}+\eps).
\end{align*}
We removed  a short stretch of the range of $x^k_1$ from the $k$-th stage of $f$ and added this short interval as an extra arm to $\ul f$ at the $\ell$-th stage. 

Let $a= 1 - \eps / (t_{k} - t_{k-1})$ so that $\eps = (t_{k} - t_{k-1})(1-a)$.
It is easy to check that
\begin{align}\label{m5.6}
\ul z^{k} _j &= z^{k} _j /a \quad \text{  for  } 2\leq j \leq m_{k},  \\
\ul z^{k} _1 &= z^{k} _1 , \nonumber \\
\ul z^{\ell} _j &= z^{\ell} _j 
\frac{ t_{\ell} - t_{\ell-1} }{ t_{\ell} - t_{\ell-1} + \eps} \quad \text{  for  } 1\leq j\leq \ul m_{\ell}-1, \nonumber \\
\ul z^{\ell} _{\ul m_{\ell}} &= z^k_1 \frac{\eps }{ t_{\ell} - t_{\ell-1} + \eps}, \nonumber \\
\ul z^{n} _j &= z^{n} _j, \qquad n\ne k,\ell, \quad j=1,\dots, m_n.
\nonumber
\end{align}
We note that \eqref{m5.6} matches \eqref{f19.3} so 
 so $\beta(\ul f) - \beta( f) $ is given by the  formula
in \eqref{m5.7}-\eqref{m5.8}. We have shown in part (i)  that this quantity is positive for small $\eps>0$ so 
 $\beta(\ul f) - \beta( f) >0$ for small $\eps >0$, and, therefore, $f\notin \calF^M_*$.

\medskip
(iii)
Suppose that
$f(0)$ contains a vertex of degree 1. Assume without loss of generality that $x^1_1(0)$ is a vertex of degree 1. 
Consider a function $\ul f$ constructed as follows. We assume that $\ul x^k_j (t) = x^k_j(t)$ for all $k\geq 2$, $1\leq j \leq m_k$ and $t\in (t_{k-1}, t_k)$, and also for $k=1$, $2\leq j \leq m_1$ and $t\in (0, t_1)$. We replace $x^1_1$ with two functions $\ul x^1_1$ and $\ul x^1_{m_1+1}$ defined for $t\in(0,t_1)$ as follows,
\begin{align*}
\ul x^1_1 (t) = x^1_1(t_1/2 - t/2),\qquad
  \ul x^1_{m_1+1} = x^1_1(t_1/2 + t/2).
\end{align*}
We have $\ul z^1_1 = \ul z^1_{m_1+1}= z^1_1/2$ so
\begin{align*}
\beta(\ul f) - \beta( f)
= - 2 t_1 (z^1_1/2) \log (z^1_1/2)
+  t_1 z^1_1 \log z^1_1>0
\end{align*}
and, therefore, $f\notin \calF^M_*$.

Next suppose that $f(0)$ does not contain a vertex of degree 1 but $f(t)$ does for some $0<t< \zeta$. Then there is $t_k < \zeta$ such that  $f(t_k)$ contains a vertex of degree 1  and $f(t)$ does not contain a vertex of degree 1 for $t < t_k$.
We can assume without loss of generality that $x^k_1(t_k)$ is a vertex of degree 1. Let $\ell = k+1$. Then it is easy to see that \eqref{f19.6} is satisfied. We can now follow the proof given in part (i) to conclude that $f\notin \calF^M_*$.

\medskip
(iv)
We will  prove that the number of disjoint components of $f(t)$ cannot decrease. 
Suppose that the number of disjoint components of $f(t)$ for $t\in (t_k, t_{k+1})$ is smaller than the number of disjoint components of $f(t)$ for $t\in (t_{k-1}, t_{k})$.  Some functions $x^k_i$ and $x^k_j$ must meet at time $t_k$, i.e., for some $i\ne j$, $x^k_i(t_k) = x^k_j(t_k)$. We can choose labels so that $x^k_1(t_k) = x^k_2(t_k)$.
Then  the functions $x^k_j$, $j=1,\dots, m_k$, and $x^{k+1}_i$, $i=1,\dots,m_{k+1}$, can be labeled in such a way that \eqref{f19.6} is satisfied. The rest of the proof given in part (i) applies verbatim so $f\notin \calF^M_*$.

It remains to show that the number of disjoint components of $f(0)$ is $M$. Suppose that this not the case. 
The proofs of part (i) and (iii) imply that an arm cannot stop growing before time $\zeta$. It follows that the set $f(\zeta-) := \bigcup_{t < \zeta} f(t)$ is non-empty.
We choose any point $v \in S\setminus f(\zeta-)$ and add it to $f(0)$ to create a new flooding. More formally, $g(t) = f(t) \cup \{v\}$ for all $t\leq \zeta$. Note that the number of sources of $g$ is equal to or less than $M$.  The flooding $g$ has a dormant arm so  $\beta(g)< \beta^*(S,M)$ according to part (ii). Since $\beta(f) =\beta(g)$, we conclude that $f\notin \calF^M_*$.

\medskip
(v)
Suppose that $f(t) $ contains a loop  for some $t< \zeta$. 
Then there is $t_k < \zeta$ such that  $f(t_k)$ contains a loop  and $f(t)$ does not contain a loop for $t < t_k$.  It follows that some functions $x^k_i$ and $x^k_j$ must meet at time $t_k$, i.e., for some $i\ne j$, $x^k_i(t_k) = x^k_j(t_k)$. We proceed as in part (iv) to conclude that $f\notin \calF^M_*$.

\medskip
(vi) 
First  suppose that $M=1$. According to Theorem \ref{m6.1} (i), $f(0)$ is the centroid of $S$. Let $d$ be the number of connected components of $S\setminus f(0)$. For $t\geq 0$, by part (v), $f(t)$ does not contain loops so the number of connected components of $f(t) \setminus f(0)$ is  equal to $d$. Let us call these sets $f_j(t)$, $j=1,\dots, d$.
If we denote the total length of all edges in $f_j(t)$ by $|f_j(t)|$, it is elementary to see that condition \eqref{m10.1} implies that for all $i,j=1,\dots, d$ and $s,t\in(0,\zeta]$,
\begin{align*}
\frac{|f_j(t)|}{|f_i(t)|} = \frac{|f_j(s)|}{|f_i(s)|}.
\end{align*}
In particular, this holds for all $t\in(0,\zeta]$ and $s=\zeta$. Hence, since $f(0)$ is the centroid of $S=f(\zeta)$, it must be the centroid of $f(t)$ for all $t\in(0,\zeta]$. 

\medskip

Next suppose that $M>1$. Recall that for $v\in f(0)$ and $t\in[0, \zeta)$, $S_v(t)$ denotes the connected component of $f(t)$ which contains $v$, and $S_v$ is the closure of $\bigcup_{1\leq t < \zeta} S_v(t)$. Let $f_v(t)= f(t) \cap S_v$. The function $t\to f_v(t)$ satisfies most conditions in Definition \ref{m5.1} of a flooding. It does not satisfy \eqref{f5.3}. We will construct a flooding $g_v$ of $S_v$ with a single source by transforming $f_v(t)$.

Let
\begin{align}\label{m21.1}
\ol z^{v,k} = \sum_{1\leq j \leq m_k, e^k_j \subset S_v} z^k_j.
\end{align}
Let $t^v_0 = 0$ and for $1\leq k \leq \kf$,
\begin{align*}
t^v_k = t^v_{k-1} + ( t_k - t_{k-1}) \ol z^{v,k}.
\end{align*}
Recall the notation from  Definition \ref{m5.1}. For $j$ such that $e^k_j \subset S_v$, let
$\wh z^{v,k}_j =  z^k_j /\ol z^{v,k}$.
Let $\zeta_v =  t^v_\kf$. For $t\in[t^v_{k-1}, t^v_k]$, let
$f_v(t) = f(t_{k-1} +(t - t^v_{k-1})/\ol z^{v,k})\cap S_v$.
It is easy to check that $f_v(t)$ is a flooding of $S_v$, with $t^v_k$'s playing the role of $t_k$'s, $\wh z^{v,k}_j$'s playing the role of $z^k_j$'s, and $\zeta_v$ playing the role of $\zeta$.

For a flooding $g$ of $S_v$,
we will write $\beta(g, S_v)$ to denote the quantity in \eqref{f5.1} relative to $g$ and $S_v$. Let
\begin{align*}
\beta(f\mid S_v)  &=
-\sum_{k=1}^{\kf}\left( (t_k - t_{k-1}) 
 \sum_{1\leq j \leq m_k, e^k_j \subset S_v} z^k_j \log z^k_j\right).
\end{align*}
We have
\begin{align}\label{m13.7}
\beta(f_v, S_v) &= 
-\sum_{k=1}^{\kf}\left( (t^v_k - t^v_{k-1}) 
 \sum_{1\leq j \leq m_k, e^k_j \subset S_v} 
\wh z^{v,k}_j \log \wh z^{v,k}_j\right)\\
&= 
-\sum_{k=1}^{\kf}\left( (t_k - t_{k-1}) \ol z^{v,k} 
 \sum_{1\leq j \leq m_k, e^k_j \subset S_v} ( z^k_j/\ol z^{v,k}) 
\log ( z^k_j/\ol z^{v,k})\right)\nonumber\\
&= 
-\sum_{k=1}^{\kf}\left( (t_k - t_{k-1}) 
 \sum_{1\leq j \leq m_k, e^k_j \subset S_v}  z^k_j \log  z^k_j\right)\nonumber\\
&\qquad+\sum_{k=1}^{\kf}\left( (t_k - t_{k-1}) 
 \sum_{1\leq j \leq m_k, e^k_j \subset S_v}  z^k_j 
\log \ol z^{v,k}\right)\nonumber\\
&= \beta(f\mid S_v)
+\sum_{k=1}^{\kf}\left( (t_k - t_{k-1}) 
  \ol z^{v,k} \log \ol z^{v,k}\right).\nonumber
\end{align}

Given a flooding $h$ of $S_v$, we  define a flooding $f^h$ of $S$ by 
\begin{align*}
f^h(t) 
= (f(t) \cap  S^c_v) )
\cup h( t^v_{k-1} + (t - t_{k-1})\ol z^{v,k} )
\end{align*}
for
$t \in [t _{k-1}, t_k]$ and $1\leq k \leq \kf$.
It is elementary to check that $(f^h)_v = h$.
Hence, we can apply \eqref{m13.7} to see that
\begin{align*}
\beta(f^h) &= \beta(h,S_v) - \sum_{k=1}^{\kf}\left( (t_k - t_{k-1}) 
  \ol z^{v,k} \log \ol z^{v,k}\right) + \beta(f^h\mid S_v^c)\\
&= \beta(h,S_v) - \sum_{k=1}^{\kf}\left( (t_k - t_{k-1}) 
  \ol z^{v,k} \log \ol z^{v,k}\right) + \beta(f\mid S_v^c).
\end{align*}
Note that only the term $\beta(h,S_v)$ depends on $h$ on the right hand side.
We see that $f_v\in \calF^1_*$ relative to $S_v$ because otherwise we could find a flooding $h \in \calF^1_*$ of $S_v$ with $\beta(h,S_v) > \beta(f_v,S_v)$, and then we would have $\beta(f^h) > \beta(f^{f_v}) = \beta(f)$, which is impossible because $f\in \calF^M_*$.

Since $f_v\in \calF^1_*$ relative to $S_v$, the first part of the proof of part (vi), for $M=1$, implies that $v$ is the centroid of $f_v(t)$ for all $t$. Part (vi) of the theorem follows easily.

\medskip
(vii) 
First assume that $y_1 = y_2$.
We will use the notation from part (vi) of the proof.
It follows from \eqref{m13.7} that
\begin{align}
\beta(f) &= \sum_{v\in f(0)} \beta(f\mid S_v)\nonumber\\
&= \sum_{v\in f(0)}
\left( \beta(f_v) - \sum_{k=1}^{\kf}\left( (t_k - t_{k-1}) 
  \ol z^{v,k} \log \ol z^{v,k}\right)\right)\nonumber \\
&= \sum_{v\in f(0)} \beta(f_v)
 - \sum_{k=1}^{\kf}\left( (t_k - t_{k-1}) \sum_{v\in f(0)}
  \ol z^{v,k} \log \ol z^{v,k}\right). \label{m21.2}
\end{align}

We may construct a new flooding $\wt f$ of $S$ by replacing $\ol z^{v,k}$'s with $\wt z^{v,k}$'s satisfying $\wt z^{v,k}>0$ for all $v$ and $k$, $\sum_{v\in f(0)} \wt z^{v,k}=1$ for all $k$, and 
$\sum_{1\leq k \leq \kf} (t_k- t_{k-1})\wt z^{v,k}=\zeta_v$ for all $v\in f(0)$. More formally, there exists a flooding $\wt f$ of $S$ such that $\wt f(0) = f(0)$, the floodings of $S_v$'s induced by $\wt f$ are the same $f_v$'s as for $f$, and the quantities in \eqref{m21.1} relative to $\wt f$ are $\wt z^{v,k}$'s.
We will find the family of $\wt z^{v,k}$'s satisfying the above conditions which maximizes the last line in \eqref{m21.2}.

Note that  $\sum_{v\in f(0)} \beta(\wt f_v)= \sum_{v\in f(0)} \beta(f_v)$, by assumption. The problem of maximizing the remaining sum is mathematically equivalent to the problem of maximizing the entropy $\beta$ of a flooding for a star graph with edge lengths $\{\zeta_v, v\in f(0)\}$, and only one source at the central vertex of the star. The methods used in previous proofs show that the optimal $\wt z^{v,k}$'s are given by
$\wt z^{v,k} = \zeta_v / \zeta$ for all $v$ and $k$. Since $f$ is assumed to be in $\calF^M_*$, we must have $\ol z^{v,k} =\wt z^{v,k}= \zeta_v / \zeta$ for all $v$ and $k$. This shows that we can apply Theorem \ref{m6.1} (ii) to $f_v$'s and, therefore, part (vii) of the theorem follows in the case $y_1 = y_2$.

Next suppose that $y_1 \ne y_2$. Consider the case when there exists $v_1\in S_{y_1} \cap S_{y_2}$. Assume that $x^\kf_i(t)$ and $x^\kf_j(t)$ approach $v_1$ as $t\to \zeta$, and these functions represent an arm of $S_{y_1}$ and an arm of $S_{y_2}$. Then $z^\kf_i = z^\kf_j$, according to Lemma \ref{m21.3}. This and the fact that \eqref{m15.11} holds in the case when $y_1=y_2$ show that \eqref{m15.11} holds also in the present case. 

Since $S$ is connected, for any $y_1, y_2 \in f(0)$, there exist points $\wh y_1 = y_1, \wh y_2, \wh y_3, \dots, \wh y_j = y_2$ in $f(0)$, such that $S_{\wh y_i} \cap S_{\wh y_{i+1}} \ne \emptyset$ for all $i$.
This implies that the claim proved in the previous paragraph extends to all 
$y_1, y_2 \in f(0)$.
\end{proof}

\begin{corollary}\label{m21.4}
Suppose that $M\geq 2$, $f \in \calF^M_*$, the functions $x^k_{j_0}, x^{k+1}_{j_1}, \dots, x^{k+n}_{j_n}$ take values in the same edge of $S$, and
$x^{k+i}_{j_i}(t_{k+i}) = x^{k+i+1}_{j_{i+1}}(t_{k+i})$ for $i=0, \dots, n-1$. Then $z^k_{j_0}= z^{k+1}_{j_1}= \dots= z^{k+n}_{j_n}$.
\end{corollary}

\begin{proof}
The corollary follows easily from Theorem \ref{f22.1} (vii).
\end{proof}

Suppose that the metric graph $S$ is a tree and $x\in S$. For $y\in S$, $y\ne x$, let $ D^x_y$ be the 
set of all $z\in S$ such that the geodesic from $x$ to $z$ contains $y$.

We will say $\{S_j\}_{1\leq j \leq k}$ is a partition of a metric graph $S$ if every $S_j$ is a closed metric tree, $\bigcup_{1\leq j \leq k} S_j = S$, and interiors of $S_j$'s are non-empty and disjoint.

\begin{proposition}\label{m14.1}
(i) If $S$ is a metric tree and $v\in S$  then 
\begin{align}\label{a2.5}
\sup\{\beta(f): f\in \calF^1, f(0) =v\}=  -\zeta 
+\zeta \log \zeta- \int_S \log |D^v_y| dy .
\end{align}

(ii) Suppose that $S$ is a metric tree, $M>1$, 
and $v_1, v_2, \dots, v_M$ are fixed distinct points in $S$.
Then
\begin{align}\label{m29.2}
\sup_f \beta(f)= 
\sup_{S_1, S_2, \dots, S_M}
 \left( -\zeta +\zeta \log \zeta - \sum_{j=1}^M \int_{S_j} \log |D^{v_j}_y| dy \right),
\end{align}
where the supremum on the left hand side is taken over functions $f\in \calF^M$ such that $f(0) = \{v_1, v_2, \dots, v_M\}$. The supremum on the right hand side  is taken over all  partitions $\{S_j\}_{1\leq j \leq M}$ of $S$, such that each $v_j$ is in the interior of $S_j$.
\end{proposition}

\begin{proof}
(i)
Let $v$ be a fixed point in $S$.
We define a partial order on $S$ by declaring that $y \prec z$ if and only if $y$ lies on the geodesic between $v$ and $z$.
Suppose that $f$ is such that $f(0)=v$ and  $\beta(f) = \sup\{\beta(f): f\in \calF^1, f(0) =v\}$. 
Let $S^k_j = x^k_j((t_{k-1} , t_k))$ for $1\leq k \leq \kf$ and $1\leq j \leq m_k$ and note that $|S^k_j| = z^k_j (t_k - t_{k-1})$. We will write $S^{k_1} _{j_1} \prec S^{k_2} _{j_2}$ if  $x^{k_1} _{j_1}(s) \prec x^{k_2} _{j_2}(t)$ for all
$s\in (t_{k_1-1}, t_{k_1})$ and $t\in (t_{k_2-1}, t_{k_2})$.
Let 
\begin{align*}
\gamma^{k_1}_{j_1} = \sum_{(k_2,j_2): S^{k_1} _{j_1} \prec S^{k_2} _{j_2}}
\left|S^{k_2} _{j_2}\right|
=
\sum_{(k_2,j_2): S^{k_1} _{j_1} \prec S^{k_2} _{j_2}} 
z^{k_2}_{j_2} (t_{k_2} - t_{k_2-1}) .
\end{align*}

We have
\begin{align}\label{a1.3}
\int_S &\log |D^v_y| dy
=
\sum_{1\leq k \leq \kf} \sum_{1\leq j \leq m_k}
\int_0^{|S^k_j|}\log (u + \gamma^k_j) du\\
&=
\sum_{1\leq k \leq \kf} \sum_{1\leq j \leq m_k}
\left((|S^k_j| + \gamma^k_j) \log (|S^k_j| + \gamma^k_j)-(|S^k_j| + \gamma^k_j) 
-  \gamma^k_j \log \gamma^k_j + \gamma^k_j\right)\nonumber \\
&= -\zeta +
\sum_{1\leq k \leq \kf} \sum_{1\leq j \leq m_k}
\left((|S^k_j| + \gamma^k_j) \log (|S^k_j| + \gamma^k_j)
-  \gamma^k_j \log \gamma^k_j \right) \nonumber\\
&=-\zeta +
\sum_{1\leq k \leq \kf} \sum_{1\leq j \leq m_k}
|S^k_j|
\log \left(\frac{|S^k_j| + \gamma^k_j}
{ \sum_{1\leq i \leq m_k} |S^k_i| + \gamma^k_i  }\right)
\nonumber \\
&\quad +
\sum_{1\leq k \leq \kf} \sum_{1\leq j \leq m_k}
\left(|S^k_j|  
\log \left(\sum_{1\leq i \leq m_k} |S^k_i| + \gamma^k_i  \right)
+\gamma^k_j
\log \left(|S^k_j| + \gamma^k_j\right)
-  \gamma^k_j \log \gamma^k_j\right).\nonumber
\end{align}
Theorem \ref{f22.1} (vii) implies that $|S^k_j| / \gamma^k_j = |S^k_i| / \gamma^k_i$ for all $1 \leq i,j \leq m_k$, so for all $j$ in this range,
\begin{align*}
\frac{|S^k_j| + \gamma^k_j}{\gamma^k_j}
= 
\frac{\sum_{1\leq i \leq m_k} (|S^k_i| + \gamma^k_i)}
{\sum_{1\leq i \leq m_k}  \gamma^k_i},
\end{align*}
and, therefore,
\begin{align*}
\gamma^k_j
\log \left(|S^k_j| + \gamma^k_j\right)
-  \gamma^k_j \log \gamma^k_j
&= \gamma^k_j
\log \left( \frac{|S^k_j| + \gamma^k_j}{\gamma^k_j} \right)
= \gamma^k_j
\log \left( \frac{\sum_{1\leq i \leq m_k} (|S^k_i| + \gamma^k_i)}
{\sum_{1\leq i \leq m_k} \gamma^k_i} \right)\\
&= \gamma^k_j
\log \left(\sum_{1\leq i \leq m_k} (|S^k_i| + \gamma^k_i) \right)
- \gamma^k_j
\log \left( \sum_{1\leq i \leq m_k} \gamma^k_i \right).
\end{align*}
We use this and \eqref{a1.3} to see that
\begin{align}\label{a2.1}
\int_S &\log |D^v_y| dy
=-\zeta +
\sum_{1\leq k \leq \kf} \sum_{1\leq j \leq m_k}
|S^k_j|
\log \left(\frac{|S^k_j| + \gamma^k_j}
{ \sum_{1\leq i \leq m_k} |S^k_i| + \gamma^k_i  }\right) \\
&\quad +
\sum_{1\leq k \leq \kf} \sum_{1\leq j \leq m_k}
|S^k_j|  
\log \left(\sum_{1\leq i \leq m_k} |S^k_i| + \gamma^k_i  \right)
\nonumber  \\
&\quad +
\sum_{1\leq k \leq \kf} \sum_{1\leq j \leq m_k}
\left(\gamma^k_j
\log \left(\sum_{1\leq i \leq m_k} (|S^k_i| + \gamma^k_i) \right)
- \gamma^k_j
\log \left( \sum_{1\leq i \leq m_k} \gamma^k_i \right)\right)\nonumber\\
&=-\zeta +
\sum_{1\leq k \leq \kf} \sum_{1\leq j \leq m_k}
|S^k_j|
\log \left(\frac{|S^k_j| + \gamma^k_j}
{ \sum_{1\leq i \leq m_k} |S^k_i| + \gamma^k_i  }\right) \nonumber \\
&\quad +
\sum_{1\leq k \leq \kf} \sum_{1\leq j \leq m_k}
(|S^k_j| + \gamma^k_j)  
\log \left(\sum_{1\leq i \leq m_k} |S^k_i| + \gamma^k_i  \right)
\nonumber  \\
&\quad -
\sum_{1\leq k \leq \kf} \sum_{1\leq j \leq m_k}
 \gamma^k_j
\log \left( \sum_{1\leq i \leq m_k} \gamma^k_i \right).\nonumber
\end{align}
We have $\gamma^\kf _j = 0$ for all $j$, so
\begin{align}\label{a1.2}
\sum_{1\leq j \leq m_{\kf}}
\gamma^{\kf}_j \log \gamma^{\kf}_j
=0.
\end{align}
Note that, for $2 \leq k \leq \kf$,
\begin{align*}
\sum_{1\leq j \leq m_k}
(|S^k_j| + \gamma^k_j)  
\log \left(\sum_{1\leq i \leq m_k} |S^k_i| + \gamma^k_i  \right)
=
 \sum_{1\leq j \leq m_{k-1}}
 \gamma^{k-1}_j
\log \left( \sum_{1\leq i \leq m_{k-1}} \gamma^{k-1}_i \right).
\end{align*}
This and \eqref{a1.2} imply that all sums on the last two lines of \eqref{a2.1} cancel except for
\begin{align*}
\sum_{1\leq j \leq m_1}
(|S^k_1| + \gamma^1_j)  
\log \left(\sum_{1\leq i \leq m_1} |S^1_i| + \gamma^1_i  \right)
= \zeta \log \zeta.
\end{align*}
Hence, \eqref{a2.1} is now reduced to
\begin{align}\label{a2.2}
\int_S \log |D^v_y| dy
&=-\zeta +\zeta \log \zeta +
\sum_{1\leq k \leq \kf} \sum_{1\leq j \leq m_k}
|S^k_j|
\log \left(\frac{|S^k_j| + \gamma^k_j}
{ \sum_{1\leq i \leq m_k} |S^k_i| + \gamma^k_i  }\right) .
\end{align}
We invoke Theorem \ref{f22.1} (vii) once again to claim that
\begin{align}\label{a2.3}
\frac{|S^k_j| + \gamma^k_j}
{ \sum_{1\leq i \leq m_k} |S^k_i| + \gamma^k_i  }
= z^k_j.
\end{align}
Recall that $|S^k_j| = (t_k - t_{k-1}) z^k_j$. We combine this observation with \eqref{a2.2}-\eqref{a2.3} to conclude that 
\begin{align*}
\int_S \log |D^v_y| dy
&=-\zeta +\zeta \log \zeta +
\sum_{1\leq k \leq \kf} (t_k - t_{k-1}) \sum_{1\leq j \leq m_k}
z^k_j
\log z^k_j \\
&= -\zeta +\zeta \log \zeta  -\beta(f) .
\end{align*}
This proves part (i) of the proposition.

(ii)
We will only sketch the proof. For fixed $S_1, S_2, \dots, S_M$, the rates at which the length covered by the optimal flooding in each of these sets are proportional to the total lengths of these graphs, i.e., $|S_1|, |S_2|, \dots, |S_M|$, by Theorem \eqref{f22.1} (vii). Hence, we can optimize $f$ separately in each $S_j$. As a result, we obtain a formula analogous to 
\eqref{a2.5} for each $S_j$ and then we take the sum. The fact that we end up with an extra term $-\zeta + \zeta \log \zeta$ which does not depend on $M$ can be best understood as a normalization corresponding to the denominator $(n\zeta)!$ in the formula for $p_n$ in Remark \ref{a2.6} below.
\end{proof}

\begin{remark}\label{a2.6}
The proof of Theorem \eqref{m14.1} (i) consists of a transformation of the right hand side of \eqref{a2.5} into the formula \eqref{f5.1} defining $\beta(f)$,
with help from Theorem \ref{f22.1} (vii), which provides information about the speed of arm growth for the optimal flooding. Our calculation is not very illuminating so let us point out that \eqref{a2.5} can be derived from formula (6.1) in \cite{twinp}. 
Then the formula in \eqref{a2.5} arises in a natural way as the limit of  approximating Riemann sums.
In other words, one can apply \cite[(6.1)]{twinp} to subdivision graphs and then  pass to the limit with the number of vertices. This alternative argument is based on the observation that if $p_n$ is the probability that a random labeling has a single peak at $v$ in the $n$-th stage subdivision graph then $p_n \approx \exp(n\beta(f))/(n\zeta)!$, for $f\in \calF^1_*$. The rigorous version of this alternative proof seems to be longer than the proof presented above.
\end{remark}

\section{Examples}\label{sec:example}

We will say that a metric tree $S$ is a $(d+1)$-regular tree of depth $n\geq 1$ for $d\geq 2$ if all vertices have degree $d+1$ or 1, all 
edges have length 1 and there exists $v_1\in S$ (the root) such that 
the distance from $v_1$ to every leaf (vertex of degree 1) is equal to $n$.

\begin{theorem}

Suppose that $S$ is the $(d+1)$-regular tree of depth $n\geq 2$, for some $d\geq 2$. If $f \in \calF^2_*$ then $f(0)$ contains the root  of $S$ and an adjacent vertex. 
\end{theorem}

\begin{proof}
It follows from Theorem \ref{f22.1} (iii) and the assumptions that $S$ is a tree and $f\in\calF^2_*$ that there exists a unique $v_0\in S$ such that $v_0\notin f(t)$ for all $t< \zeta$ and $v_0$ is not a leaf. The point $v_0$ lies on the geodesic between the two points in $f(0)$. It is easy to see that $v_0$ is not a vertex. Let $S_1$ and $S_2$ be the closures of the two connected components of $S\setminus \{v_0\}$.

Recall 
the definition of $D^x_y$ stated before Proposition \ref{m14.1} and let $v_1$ be the root of $S$.
We will say that $\wt S$ is a full branch of depth $ k$ if 
$\wt S = D^{v_1}_y$ for some $y \ne v_1$ and the distance from $y$ to a leaf of $\wt S$ (different from $y$) is in the range $[k, k+1)$.

Suppose without loss of generality that $v_1\in S_1$. It is easy to see that $v_1$ is the centroid of $S_1$. 
The set
$S_2$ is  a full branch with depth $k$ for some $0\leq k < n$. Let $v_2$ be the centroid of $S_2$ and let $v_3$ be the leaf of $S_2$ that lies on the geodesic from $v_1$ to $v_2$. Note that if $1\leq k < n$ then $v_2$  is a branch point of $S$ that lies at the distance $k$ from all leaves of $S_2$ except  $v_3$. Recall that $\rho$ denotes the usual metric on $S$.
Let  $b  = \rho(v_2, v_3)$  and note that $0< b < 1$.

We will prove that the depth $k$ of $S_2$ is equal to $n-1$.
Suppose otherwise, i.e., assume that that $k < n-1$. 
Let $v_4$ lie on the geodesic from $v_3$ to $v_1$, with $\rho(v_3, v_4) = 1$, and let $v_5$ be the vertex of $S$ between $v_3$ and $v_4$. Let $\wh S_2$ be the full branch of depth $k+1$, with centroid at $v_5$ and such that $v_4$ is the closest point to $v_1$ among all points in $\wh S_2$. Let $\wh S_1$ be the closure of $S\setminus \wh S_2$. See Fig.~\ref{flood1}.

\begin{figure} \includegraphics[width=0.9\linewidth]{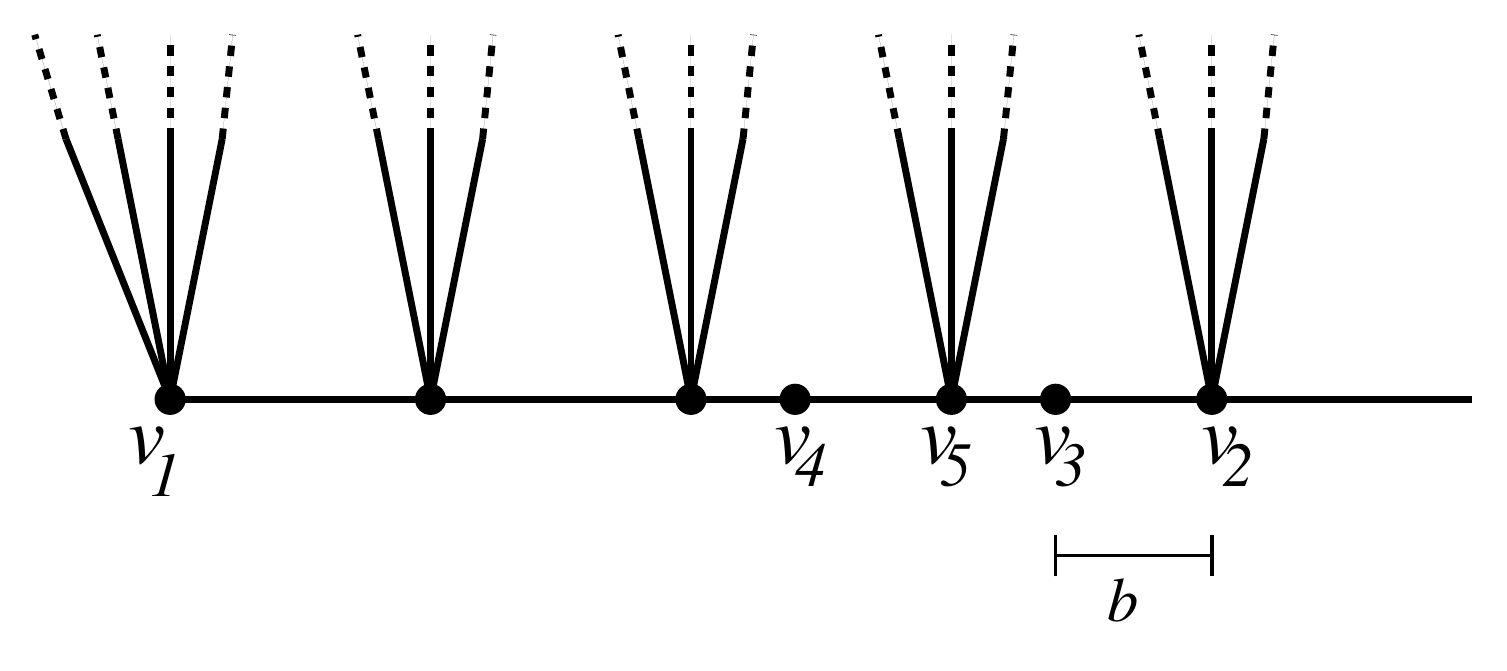}
\caption{
Illustration of the positions of $v_1, \dots, v_5$ on the $(d+1)$-regular tree for $d= 4$ and the depth $k$ of $S_2$ equal to $ n-4$.
}
\label{flood1}
\end{figure}

For any subtree $S_*$ of $S$, let $D^{v, S_*}_y$ be defined in the same way as $D^{v}_y$ but relative to $S_*$.
We will show that
\begin{align}\label{m15.1}
&\exp \left( -\zeta+\zeta \log \zeta -  \int_{S_1} \log |D^{v_1, S_1}_y| dy 
-  \int_{S_2} \log |D^{v_2, S_2}_y| dy \right) \\
&\quad < \exp \left( -\zeta+\zeta \log \zeta -  \int_{\wh S_1} \log |D^{v_1, \wh S_1}_y| dy 
-  \int_{\wh S_2} \log |D^{v_5, \wh S_2}_y| dy \right) . \nonumber
\end{align}
If \eqref{m15.1} holds then,
in view of Proposition \ref{m14.1} (ii),  there exists $g\in \calF^2$ such that  $\beta(g) > \beta(f)$ and, hence,  $f \notin \calF^2_*$.
This  implies that $k=n-1$, i.e.,  $S_2$ is a   full branch of depth $n-1$.

To show \eqref{m15.1},
it will suffice to prove that the following quantity is strictly positive,
\begin{align*}
I:=
\int_{S_1} \log |D^{v_1, S_1}_y| dy
+\int_{S_2} \log |D^{v_2, S_2}_y| dy
- \int_{\wh S_1} \log |D^{v_1,\wh S_1}_y| dy
-\int_{\wh S_2} \log |D^{v_5,\wh S_2}_y| dy.
\end{align*}
 ``Most'' differentials
which have the form $\log |D^{v_*, S_1}_y| dy$ or $\log |D^{v_*, S_2}_y| dy$ in the first two integrals
match and, therefore, cancel some differentials in the last
two integrals, because of the sign difference. 
Note that  a $y\in S$ which appears in the first integral may 
appear either in the third or the fourth integral, and then the corresponding
differentials  cancel.
The set of $y$'s not associated with this cancellation of differentials
consists  of all $y$'s on the geodesic between $v_1$ and $v_2$, and only these points. Let $G_1, G_2, G_3$ and $G_4$ be the geodesics between the following pairs of points: $(v_1,v_3), (v_3, v_2), (v_1,v_4) $ and $(v_4, v_2)$. 
We conclude  that
\begin{align}\label{m15.9}
I=
\int_{G_1} \log |D^{v_1, S_1}_y| dy
+\int_{G_2} \log |D^{v_2, S_2}_y| dy
- \int_{G_3} \log |D^{v_1,\wh S_1}_y| dy
-\int_{G_4} \log |D^{v_5,\wh S_2}_y| dy.
\end{align}

If the distance of $y\in S$ from the closest leaf of $S$ is $a$ and $m=\lfloor a \rfloor$ then 
\begin{align}\label{m29.1}
 |D^{v_1, S}_y| = a- m +\sum_{j=1}^m d^j
= a-m + \frac d {d-1} (d^m -1)
= a-m + \frac {d^{m+1}} {d-1} - \frac d {d-1}.
\end{align}

If $y \in G_1$ then the distance from $y$ to the closest leaf of $S$ is
in the range $(k + b, n)$
and $ D^{v_1, S_1}_y = D^{v_1, S}_y \setminus D^{v_1, S}_{v_3}$. 
This and \eqref{m29.1} imply that if $y \in G_1$ and the distance of $y$ from the closest leaf of $S$ is in the range $(j-1, j)$ for some $k+2\leq j \leq  n$ then
\begin{align}\label{m15.2}
|D^{v_1, S_1}_y| &= |D^{v_1, S}_y| - |D^{v_1, S}_{v_3}|\\
&= n-j+1 -\rho(y, v_1) + \frac {d^j} {d-1} - \frac d {d-1}
- b - \frac {d^{k+1}} {d-1} + \frac d {d-1} \nonumber \\
& = n-j+1 -\rho(y, v_1) + \frac {d^j} {d-1} 
- b - \frac {d^{k+1}} {d-1}.\nonumber
\end{align}
If $y \in G_1$ and the distance of $y$ from the closest leaf of $S$ is in the range $(k+b, k+1)$  then
\begin{align}\label{m15.3}
|D^{v_1, S_1}_y| = \rho(y, v_3).
\end{align}

If $y \in G_2$  then
\begin{align}\label{m15.4}
|D^{v_2, S_2}_y| = \rho(y, v_3).
\end{align}

If $y \in G_3$ then the distance from $y$ to the closest leaf of $S$ is
in the range $(k+1 + b, n)$. We have the following formula, analogous to \eqref{m15.2}.
If  $y \in G_3$ and the distance of $y$ from the closest leaf of $S$ is in the range $(j-1, j)$ for some $k+3\leq j \leq  n$ then
\begin{align}\label{m15.5}
|D^{v_1, \wh S_1}_y| 
& = n-j+1 -\rho(y, v_1) + \frac {d^j} {d-1} 
- b - \frac {d^{k+2}} {d-1}.
\end{align}
If $y \in G_3$ and the distance of $y$ from the closest leaf of $S$ is in the range $(k+1+b, k+2)$  then
\begin{align}\label{m15.6}
|D^{v_1, \wh S_1}_y|
= \rho(y, v_4).
\end{align}

If $y \in G_4$ and $y$ lies between $v_1$ and $v_4$  then
\begin{align}\label{m15.7}
|D^{v_5, \wh S_2}_y| = \rho(y, v_4).
\end{align}
If $y \in G_4$ and $y$ lies between $v_5$ and $v_2$ then we obtain using \eqref{m29.1},
\begin{align}\label{m15.8}
|D^{v_5, \wh S_2}_y| &= |D^{v_1, S}_{y}| = 
\rho(y, v_2) + \frac {d^{k+1}} {d-1} - \frac d {d-1}.
\end{align}

We combine \eqref{m15.9} and \eqref{m15.2}-\eqref{m15.8} to see that
\begin{align*}
I &=   \sum_{j=k+2}^{n}  \int_0^1 \log\left(s + \frac {d^{j}} {d-1}  -b - \frac {d^{k+1}} {d-1} \right) ds
+ \int_0^{1-b} \log s \, ds
+ \int_0^{b} \log s \, ds \\
& \quad -  \sum_{j=k+3}^{n}  \int_0^1 \log\left(s + \frac{d^{j}}{d-1} -b -\frac{d^{k+2}}{d-1}\right) ds
- \int_0^{1-b} \log s \, ds
- \int_0^{b} \log s \, ds\\
& \quad - \int_0^{1} \log \left(s + \frac {d^{k+1}} {d-1} - \frac d {d-1} \right) \, ds\\
&=  \int_0^1 \log\left(s +\frac {d^{k+2}} {d-1}  -b - \frac {d^{k+1}} {d-1}\right) ds \\
& \quad +  \sum_{j=k+3}^{n}  \int_0^1 
\left(\log\left(s + \frac {d^{j}} {d-1}  -b - \frac {d^{k+1}} {d-1} \right) - \log\left(s + \frac {d^{j}} {d-1}  -b - \frac {d^{k+2}} {d-1}\right)
\right) ds\\
& \quad - \int_0^{1} \log \left(s + \frac {d^{k+1}} {d-1} - \frac d {d-1} \right) \, ds\\
&=  \int_0^1 \left(\log\left(s +d^{k+1}  -b \right) 
- \log \left(s + \frac {d^{k+1}} {d-1} - \frac d {d-1} \right) \right) ds \\
& \quad +  \sum_{j=k+3}^{n}  \int_0^1 
\left(\log\left(s + \frac {d^{j}} {d-1}  -b - \frac {d^{k+1}} {d-1} \right) - \log\left(s + \frac {d^{j}} {d-1}  -b - \frac {d^{k+2}} {d-1}\right)
\right) ds\\
& > 0.
\end{align*}
The last inequality is due to the fact that all integrands in the ultimate representation of $I$ are positive.
We conclude that \eqref{m15.1} holds and this shows that, as we have pointed out, $f \notin \calF^2_*$. 
Hence, $S_2$ must be a full branch of depth $n-1$. In other words,  vertex $v_2$ must be a neighbor of $v_1$. 
\end{proof}

\begin{example}\label{f16.1}
Suppose that $S$ is a circle from the topological point of view and consider any integer $M\geq 1$.
The expression $ \sum_{1\leq j \leq m_k} z^k_j |\log z^k_j|$ which appears in \eqref{f5.1} is a form of entropy. 
Recall from \eqref{f5.3} that $\{z^k_j, 1\leq j \leq m_k\}$ forms a probability distribution.
For given $m_k$, the entropy is maximized if this is the uniform distribution on $[m_k]$. Assuming that $\{z^k_j, 1\leq j \leq m_k\}$ form the uniform distribution on $[m_k]$, the value of the entropy is strictly increasing in $m_k$. Since $f(0) $ consists of $M$ points, the highest possible value for $m_k$ is $2M$. It is easy to see that there is only one function $f$, up to rotations of $S$, with the property that $m_k = 2M$ and $\{z^k_j, 1\leq j \leq m_k\}$ is the uniform distribution for all $k$. For this function $f$, $\kf=1$, the $M$ points comprising $f(0)$ are equally spaced around the circle and $z^1_j=1/(2M)$ for all $j$.
We conclude that for a fixed $M$, $\beta(f)$ is maximized at this function $f$ and its rotations.
\end{example}

\begin{example}
If $S$ is a line segment from the topological point of view and $M\geq 1$ then analysis similar to that in Example \ref{f16.1} shows that $\beta(f)$ is maximized by a single function $f$ such that $f(0)$ consists of $M$ points at the distances $(i+1/2)\zeta / M$, $i=0,1, \dots, M-1$, from one of the endpoints of $S$. 
We also have $ \kf=1$ and $z^1_j=1/(2M)$ for all $j$.
\end{example}

\begin{example}\label{f16.2}

Suppose that $\calG$ is a star graph with four vertices and three edges $e_1, e_2$ and $e_3$.  
Let $v_*$ be the center of the star, i.e., the vertex adjacent to the three edges 
$e_1, e_2$ and $e_3$.  
We will simplify the notation and write $r_1, r_2$ and $r_3$ instead of $r_{e_1}, r_{e_2}$ and $r_{e_3}$.
Let $p_k = r_k/(r_1+r_2 + r_3)$ for $k=1,2,3$. 
Consider floodings with one source, i.e., $M=1$.
Assume without loss of generality that $r_1 \geq r_2 \geq r_3$. We will argue that
Theorem \ref{m6.1} implies the following assertions.

\medskip

(i) If $r_1 \leq r_2 + r_3$ 
then $\beta(f)$ is maximized by a single function $f$ such that $f(0)=v_*$, $\kf=1$ and $ z^1_j=p_j$ for  $j=1,2,3$. 

\medskip

(ii) Suppose that  $r_1 > r_2 + r_3$. Let $q_1 = 1/2$ and $q_j = r_j/(2(r_2 +r_3))$ for $j=2,3$. 
Then the functional $\beta(f)$ is maximized by a single function $f$ with the following properties: 

(a) $\kf = 2$, 

(b) $f(0)$ lies on the edge $e_1$ at the distance $(r_1-r_2 - r_3) /2$ from
 $v_*$, 

(c) $t_0 = 0 < t_1 = (r_1 - r_2-r_3) /2 < t_2 = \zeta$,

(d) $m_1 = 2$ and $ z^1_j=1/2$ for  $j=1,2$,

(e) $m_2 = 3$ and $z^2_j=q_j$ for  $j=1,2,3$.

\medskip

The point $f(0)$ must be the centroid of $S$ by
Theorem \ref{m6.1} (i). This implies easily the explicit formulas
for the position of $f(0)$ given above. 
The remaining assertions follow  from part (ii) of  Theorem \ref{m6.1}.

\end{example}

\begin{example}\label{m15.10}
We will examine the same graph as in Example \ref{f16.2} but we will be concerned with floodings with two sources, i.e., we will take $M=2$.   Explicit calculations (performed mostly with Mathematica) do not seem to yield interesting formulas so we will limit ourselves to some special cases and partial results. Still we believe that the example illustrates well some properties of floodings.

Recall the setup and notation from Example \ref{f16.2}. Consider $f \in \calF^2_*$. The set $f(0)$ contains two points; we will call them $y_1$ and $y_2$.
It follows from Theorem \ref{f22.1} that only the following three scenarios are possible.

(1)  $y_1$ is in the interior of an edge and $y_2=v_*$.

(2) $y_1$ and $y_2$ belong to the interior of the same edge.

(3) $y_1$ and $y_2$ belong to the interiors of different edges.

We will show that all three cases arise for some values of $r_1, r_2$ and $r_3$.

For $j=1,2$, let $S_j(t)$ be the connected component of $f(t)$ which contains $y_j$, for $t\in[0, \zeta)$. Let $S_j$ be the closure of $\bigcup_{1\leq t < \zeta} S_j(t)$. Assume without loss of generality that $v_*\in S_2$. Let $w_1\in S$ be the point
which does not belong to $f(t)$ for $t<\zeta$ and lies between $y_1$ and $y_2$. In other words, $w_1$ belongs to the intersection of the boundaries of $S_1$ and $S_2$. Let $w_2$ be the endpoint of $e_1$ opposite to $v_*$.

\medskip

(1) 
Suppose that $r_1=r_2=r_3 >0$.
Theorem \ref{f22.1} (ii) says that $f$ has no dormant arms so $S_2$ must contain a neighborhood of $v_*$. By Theorem \ref{f22.1} (iv), $S_1$ must be contained in one of the edges. Since the other two edges belong to $S_2$, its centroid $y_2$ is located at $v_*$. The centroid of $S_1$ must be an interior point of one of the edges. This shows that if $r_1= r_2=r_3$ then (1) holds.

We will analyze this case in more detail. Note that $\kf = 1$. The set $S_1$ is a line segment contained in one of the edges. By symmetry, we can suppose that $S_1 \subset e_1$. The point $y_1$ lies in the middle of $S_1$ by Theorem \ref{f22.1} (vi).  This and Corollary \ref{m21.4} imply that  $\rho(w_2, y_1) = \rho(y_1,w_1) = \rho(w_1, v_*) = r_1/3$. In other words, $y_1$ lies at the distance $2r_1/3$ from $v_*$. 

Given the positions of $y_1, y_2$ and $w_1$, it follows easily from Theorem \ref{f22.1} (vii) that the arms of the flooding growing from $y_1$ move at the rate $1/9$, and the arms growing from $y_2$ move at the speeds $1/9$, $1/3$ and $1/3$. 

\medskip

(2) 
Suppose that $r_1 > 128(r_2+r_3) $ and $r_2 \geq r_3$. We will argue that $y_1$ and $y_2$ lie in the interior of the edge $e_1$.

Consider other options. The points $y_1$ and $y_2$ cannot both lie outside the interior of $e_1$ because then $y_1$ would not be the centroid of $S_1$ or $y_2$ would not be the centroid of $S_2$.

Next suppose that $y_2 = v_*$ and $y_1$ lies on $e_1$. Then, using the notation and reasoning as in case (1) above, we would have  $\rho(w_2, y_1) = \rho(y_1,w_1) = \rho(w_1, v_*) = r_1/3$.
But then $y_2$ would not be the centroid of $S_2$ because $r_1 > 128(r_2+r_3) $. So it is impossible that $y_2 = v_*$ and $y_1$ lies on $e_1$.

The next case is when $y_1$ lies on $e_1$, $y_2$ lies on $e_3$ and the point $w_1$ at the intersection of the boundaries of $S_1$ and $S_2$ lies in $e_1$. Since $r_2 \geq r_3$, $y_2$ cannot be the centroid of $S_2$ so this observation eliminates this case.

Suppose that $y_1$ lies on $e_1$, $y_2$ lies on $e_2$ and the point $w_1$ at the intersection of the boundaries of $S_1$ and $S_2$ lies in $e_1$. By Corollary \ref{m21.4}, the arms emanating from $y_1$ move at constant speeds. Just like in case (1), we see that they move at the same speed, say, $z_1$. By Lemma \ref{m21.3}, the arm of  $f$ moving from $v_*$ towards $w_1$ also has the speed $z_1$. By Theorem \ref{f22.1} (vii), the speed of the arm moving along $e_2$ towards $v_*$ must be greater than or equal to $z_1$. By the same theorem, the speed of the arm moving along $e_2$ towards the  vertex  
(say, $w_3$) at the other end of $e_2$ must be also $z_1$ or greater. Multiplying the speeds by time intervals, we obtain the following inequality,
\begin{align*}
\rho(w_2, y_1) + \rho(y_1, w_1) \leq \rho(w_1, v_*) + \rho(v_*, y_2) + \rho(y_2, w_3).
\end{align*}
Since $y_2$ is the centroid of $S_2$, we must have $\rho(w_1, v_*) \leq \rho(y_2, w_3) \leq \rho(v_*, w_3) = r_2$. 
It follows that
\begin{align*}
r_1 - r_2 &= \rho(w_2, y_1) + \rho(y_1, w_1) + \rho(w_1, v_*) -r_2
\leq \rho(w_2, y_1) + \rho(y_1, w_1) \\
&\leq \rho(w_1, v_*) + \rho(v_*, y_2) + \rho(y_2, w_3) \leq 2 r_2.
\end{align*}
Hence, $r_1 \leq 3 r_2$, which is impossible in view of the assumption
that $r_1 > 128(r_2+r_3) $.

The next case to be eliminated is when $y_1$ lies on $e_1$, $y_2$ lies on $e_2$ and the point $w_1$ at the intersection of the boundaries of $S_1$ and $S_2$ lies in $e_2$. 

It is easy to see that $\kf=2$ in this case.
We will specify the edges where the functions $x^k_j$ take values. First, we declare that $x^1_1$ and $x^1_2$ represent arms emanating from $y_2$, $x^1_3$ and $x^1_4$ represent arms emanating from $y_1$, $x^2_1$ and $x^2_2$ represent arms of $S_2$, and $x^2_3, x^2_4$ and $x^2_5$ represent arms of $S_1$.
This and the following conditions identify the relationship between $x^k_j$'s and $S_n$'s uniquely,
\begin{align*}
x^1_1((t_0, t_1)) &\subset e_2 \cap \ol{y_2, w_3},\qquad
x^1_2((t_0, t_1)) \subset e_2 \cap \ol{y_2, v_*},\\
x^1_3((t_0, t_1)) &\subset e_1 \cap \ol{y_1, w_2},\qquad
x^1_4((t_0, t_1)) \subset e_1 \cap \ol{y_1, v_*},\\
x^2_1((t_1, t_2)) &\subset e_2 \cap \ol{y_2, w_3},\qquad
x^2_2((t_1, t_2)) \subset e_2 \cap \ol{y_2, v_*},\\
x^2_3((t_1, t_2)) &\subset e_1 \cap \ol{y_1, w_2},\qquad
x^2_4((t_1, t_2)) \subset e_2 ,\qquad
x^2_5((t_1, t_2)) \subset e_3 .
\end{align*}

Note that the function $z\to -z \log z$ is increasing on $(0, 1/e)$ and its
maximum on $(0,1)$ is equal to $e$ and is attained at $z=1/e$.

By Corollary \ref{m21.4}, $z^1_3 = z^2_3$. Since $y_1$ is the centroid of $S_1$, $\rho(y_1, w_2) \geq r_1/2$.
By Theorem \ref{f22.1} (vii), 
$z^1_3 \geq  (r_1/2)/\zeta$. It follows that
\begin{align}
-(t_1-t_0) z^1_3 \log z^1_3 
- (t_2-t_1) z^2_3 \log z^2_3 
&= -\rho(y_1, x^1_3(t_1)) \log z^1_3 
- \rho( x^1_3(t_1), w_2) \log z^2_3 \nonumber \\
&= -\rho(y_1, w_2)\log z^1_3
\leq -\rho(y_1, w_2)\log (r_1/(2\zeta)).\label{m22.1}
\end{align}
We use the fact that $y_1$ is the centroid of $S_1$ and
 Theorem \ref{f22.1} (vii) to see that $z^1_4 = z^1_3$ and, therefore, 
$z^1_4=z^1_3 \geq  (r_1/2)/\zeta$. Thus,
\begin{align}\label{m22.2}
-(t_1-t_0) z^1_4 \log z^1_4 
&= -\rho(y_1, v_*) \log z^1_4
\leq -\rho(y_1, v_*)\log (r_1/(2\zeta)).
\end{align}
Since $r_1 > 128(r_2+r_3) $, Corollary \ref{m21.4} and Theorem \ref{f22.1} (vii) imply that each speed $z^1_1,z^1_2, z^2_1$ and $z^2_2$ is less than $(r_2/2)/\zeta < (r_2/2)/ r_1 < 1/256$. 
We obtain
\begin{align}\label{m22.3}
&-(t_1-t_0) z^1_1 \log z^1_1 -(t_2-t_1) z^2_1 \log z^2_1 
-(t_1-t_0) z^1_2 \log z^1_2 -(t_2-t_1) z^2_2 \log z^2_2 \\
&\leq -(t_1-t_0) (1/256) \log (1/256) -(t_2-t_1) (1/256) \log (1/256) \nonumber \\
&\quad-(t_1-t_0) (1/256) \log (1/256) -(t_2-t_1) (1/256) \log (1/256) \nonumber \\
&=  -(\zeta/64) \log (1/256).\nonumber
\end{align}
Recall that the function $z\to -z \log z$ attains its
maximum on $(0,1)$  at $z=1/e$.
The interval $f([t_1,t_2])\cap e_1$ cannot be longer than the combined length of edges $e_2$ and $e_3$ because $y_1$ is the centroid of $S_1$. Hence, $t_2 - t_1 \leq 2 (r_2 + r_3)$.
We obtain
\begin{align}\label{m22.4}
&- (t_2-t_1) z^2_4 \log z^2_4 - (t_2-t_1) z^2_5 \log z^2_5\\
&\leq - (t_2-t_1) (1/e) \log (1/e) - (t_2-t_1) (1/e) \log (1/e) \nonumber \\
&\leq  (2/e) (r_2 + r_3) .\nonumber
\end{align}
Combining \eqref{m22.1}-\eqref{m22.4} yields
\begin{align}\label{m22.5}
\beta(f) &\leq 
-\rho(y_1, w_2)\log (r_1/(2\zeta)) -\rho(y_1, v_*)\log (r_1/(2\zeta))\\
&\qquad-(\zeta/2) \log (1/8) -(1/2) (r_2 + r_3) \log (1/8) \nonumber \\
&= -r_1 \log (r_1/(2\zeta))
 -(\zeta/64) \log (1/256) + (2/e) (r_2 + r_3).\nonumber
\end{align}
We have
\begin{align}\label{m22.10}
-r_1 \log (r_1/(2\zeta)) &= -r_1 \log (r_1/(2(r_1+r_2+r_3)))
\leq -r_1 \log (r_1/(2(r_1+r_1/128 )))\\
& =-r_1 \log (64/129) \leq -r_1 (5/4) \log (1/2),
\nonumber
\end{align}
\begin{align}\label{m22.11}
 -(\zeta/64) \log (1/256) & \leq -(2r_1/64) 8\log (1/2)
=-r_1(1/4) \log (1/2) ,
\end{align}
\begin{align}\label{m22.12}
 (2/e) (r_2 + r_3) & \leq (2/e) (r_1/128) 
\leq -r_1(1/(64 e)) \log (1/2) .
\end{align}
It follows from \eqref{m22.5}-\eqref{m22.12} that
\begin{align}\label{m22.13}
\beta(f) \leq - r_1 (7/4) \log  (1/2).
\end{align}

We will consider an alternative flooding that does not posses properties listed in Theorem \ref{f22.1} (and so it is not in $\calF^2_*$) but nevertheless has a higher value of $\beta$. Let $w_4$ and $w_5$ be points on $e_1$, at the distance $r_1/4$ from each end of this edge. Let $g$ be the flooding with sources $w_4$ and $w_5$, such that $z^1_1=z^1_2=z^1_3=z^1_4= 1/4$. In other words, the four arms emanating from $w_4$ and $w_5$ move at the identical speeds and completely fill $e_1$ after time $r_1$. Hence, $t_1= r_1$. The remaining part of $g$, after time $t_2$, is defined arbitrarily.
We have, using \eqref{m22.13},
\begin{align*}
\beta(g) \geq -\sum_{j=1}^4 (t_1-t_0) z^1_j \log z^1_j
= - r_1 \log(1/4) = -r_1 2 \log(1/2) > \beta(f).
\end{align*}
We have thus eliminated the case when $y_1$ lies on $e_1$, $y_2$ lies on $e_2$ and the point $w_1$ at the intersection of the boundaries of $S_1$ and $S_2$ lies in $e_2$. 

The final case to consider is when $y_1$ lies on $e_1$, $y_2$ lies on $e_3$ and  $w_1\in e_3$. This case can be dealt with exactly in the same way as the case when  
$y_1\in e_1$, $y_2\in e_2$ and  $w_1\in e_2$.

We conclude that if $r_1 > 128(r_2+r_3) $  then $y_1$ and $y_2$ lie in the interior of the edge $e_1$.

\medskip
(3) 
Suppose that $r_3 \leq r_2 \leq r_1$, $r_2 > r_1/3 + r_3$ and $r_1 > r_3/3 + r_2$.
For example, we could take $r_3 = 2/5$, $r_2 = 5/6$ and $r_1=1$.
We will show that $y_1$ and $y_2$ must belong to the interiors of two different edges.

Suppose that $y_1$ and $y_2$ belong to $e_1$. Assume without loss of generality that  $y_2=v_*$ or $y_2$  lies between $y_1$ and $v_*$.  Then the argument used in case (1) shows that 
 $\rho(w_2, y_1) = \rho(y_1,w_1) = \rho(w_1, y_2) $. These quantities are less than or equal to $r_1/3$ so $y_2$ cannot be the centroid of $S_2$, since we assumed that  $r_2 > r_1/3 + r_3$. Thus $y_1$ and $y_2$ cannot belong to $e_1$ simultaneously.

We can eliminate the possibility that $y_1$ and $y_2$ belong to $e_3$ in the analogous way, using the assumption that $r_1 > r_3/3 + r_2$.

The assumptions that $r_3 \leq r_2 \leq r_1$ and $r_2 > r_1/3 + r_3$ imply that  $r_1 > r_2/3 + r_3$. This and a reasoning similar to the one employed earlier show that $y_1$ and $y_2$ cannot belong to $e_2$ at the same time.

We conclude that $y_1$ and $y_2$  belong to the interiors of two different edges.
\end{example}

\begin{example}
Let the Cartesian product 
of path (linear) graphs with $m$ and $n$ vertices be denoted $\calG_{m,n}$.
Suppose without loss of generality that $m\geq n$, let $S$ be the metric graph corresponding to $\calG_{m,n}$ with all edges of length 1, and consider a flooding $f\in F^1_*$ of $S$. According to 
Theorem \ref{f22.1} (v), $f(t)$ does not contain loops for any $t< \zeta$. So if $t<\zeta$ then for every $1\times 1$ square in $\calG_{m,n}$, 
some part of the boundary of the square with positive length does not intersect  $f(t)$.

We will call a vertex internal if it has degree 4.
The number of non-internal vertices
is bounded by $4m$. 
Suppose that $ 32 m< t < \zeta$ and $f\in \calF^1_*$. If for an edge $e$ we have $f(t) \cap e\ne \emptyset$ then we choose
arbitrarily one of the endpoints of $e$ that belongs to $f(t) \cap e$ and we call it $v_e$. If $f(t) \cap e= \emptyset$ then we choose $v_e$ in an arbitrary way. Let $\mu(v)$ be the sum of the lengths of $f(t) \cap e$ over all $e$ such that $v=v_e$. Note that $\mu(v) \leq 4$ for all $v$. Since the total length of line segments comprising $f(t)$ is $t$, it follows that the number of vertices with $\mu(v) >0$ is greater than or equal to $t/4 > 32m/4 = 8m$. It follows that the number $N_1$ of internal vertices with $\mu(v)>0$ is at least $t/4 - 4m> t/8$. Let $Q_v$ be the $1\times 1$ square such that $v$ is  the northwest corner of $Q_v$. If $\mu(v)>0$ then the boundary of $Q_v$ cannot be totally contained in $f(t)$ or its complement. Hence the boundary of $Q_v$ contains at least two leaves of $f(t)$. Each of the two leaves can belong to the boundary of at most four squares so the number of leaves of $f(t)$ is at least $N_1/2$. This implies that that  the number of leaves of $f(t)$ is bounded below by $t/16$ for $ 32 m< t < \zeta$.  

If $f(t) \cap e$ is non-empty and strictly smaller than $e$ then we will call $e$ incomplete. If $f(t) \cap e=e$ then we will call $e$ complete. We will say that $e_1$ is an offspring of $e_2$ if $e_1$ is incomplete, $e_2$ is complete and the two edges share an endpoint. 
There are at most four incomplete edges without a parent.
Every parent has at most 6 children.
The sum of the lengths of all complete edges is at most $f(t)$. The number of leaves is equal to the number of incomplete edges and so it is bounded by $6f(t)+4$.

We conclude that the number of leaves of $f(t)$ is in the range $[t/16, 6t+4]$ for $ 32 m< t < \zeta$.  
This claim agrees  with Theorem 4.2 in \cite{twinp} at the intuitive level.

\end{example}

\section{Acknowledgments}

We are grateful to 
Omer Angel, J\'er\'emie Bettinelli, Ivan Corwin,
Ted Cox, Nicolas Curien, Michael Damron,  Dmitri Drusvyatskiy, Rick Durrett, Martin Hairer, Christopher Hoffman,
Doug Rizzolo, Timo Seppalainen, Alexandre Stauffer, Wendelin Werner and Brent Werness
for the most helpful advice.
The first author is grateful to the Isaac Newton Institute for Mathematical Sciences, where this research was partly carried out, for the hospitality and support.

\bibliographystyle{alpha}
\bibliography{isolated}

\begin{thebibliography}{BP16}

\bibitem[BP16]{twinp}
Krzysztof Burdzy and Soumik Pal.
\newblock Twin peaks.
\newblock 2016.
\newblock Preprint.

\bibitem[DZ98]{DZ}
Amir Dembo and Ofer Zeitouni.
\newblock {\em Large deviations techniques and applications}, volume~38 of {\em
  Applications of Mathematics (New York)}.
\newblock Springer-Verlag, New York, second edition, 1998.

\end{thebibliography}

\end{document}